\documentclass{amsart}
\usepackage{amsmath,amsthm,amssymb}
\usepackage{enumerate}
\usepackage{amsfonts}
\usepackage{calligra}
\usepackage{appendix}
\usepackage{mathrsfs}
\usepackage[english]{babel}
\usepackage{mathtools}
\usepackage{tensor}
\usepackage{graphicx}
\usepackage{subfig}

\def\R{\mathbb{R}}
\def\N{\mathbb{N}}

\def\T{\mathbb{T}}
\def\supp{\operatorname{supp}}
\def\dive{\operatorname{div}}
\def\curl{\operatorname{curl}}

\def\ker{\operatorname{ker}}
\newcommand{\norm}[1]{\left\lVert#1\right\rVert}
\newcommand{\abs}[1]{\left\lvert#1\right\rvert}

\newcommand{\defeq}{\mathrel{:\mkern-0.25mu=}}
\newcommand{\eqdef}{\mathrel{=\mkern-0.25mu:}}

\newtheorem{thm}{Theorem}[section]
\newtheorem{cor}[thm]{Corollary}
\newtheorem{prop}[thm]{Proposition}
\newtheorem{lem}[thm]{Lemma}

\newtheorem{assumption}[thm]{Assumption}

\theoremstyle{definition}
\newtheorem{defin}[thm]{Definition}
\newtheorem{rem}[thm]{Remark}

\numberwithin{equation}{section}

\begin{document}

\title{Proof of Taylor's conjecture on magnetic helicity conservation}


\author{Daniel Faraco}
\address{Departamento de Matem\'{a}ticas \\ Universidad Aut\'{o}noma de Madrid, E-28049 Madrid, Spain; ICMAT CSIC-UAM-UC3M-UCM, E-28049 Madrid, Spain}
\email{daniel.faraco@uam.es}
\thanks{D.F. was partially supported by ICMAT Severo Ochoa projects SEV-2011-0087 and SEV-2015-556, the grants MTM2014-57769-P-1 and MTM2014-57769-P-3 (Spain) and the ERC grant 307179-GFTIPFD. S.L. was supported by the ERC grant 307179-GFTIPFD}

\author{Sauli Lindberg}
\address{Departamento de Matem\'{a}ticas \\ Universidad Aut\'{o}noma de Madrid, E-28049 Madrid, Spain; ICMAT CSIC-UAM-UC3M-UCM, E-28049 Madrid, Spain}
\email{sauli.lindberg@uam.es}
\thanks{}


\date{}

\begin{abstract}
We prove Taylor's conjecture which says that in 3D MHD, magnetic helicity is conserved in the ideal limit in bounded, simply connected, perfectly conducting domains. When the domain is multiply connected, magnetic helicity depends on the vector potential of the magnetic field. In that setting we show that magnetic helicity is conserved for a large and natural class of vector potentials but not in general for all vector potentials. As an analogue of Taylor's conjecture in 2D, we show that mean square magnetic potential is conserved in the ideal limit, even in multiply connected domains.
\end{abstract}

\maketitle

\section{Introduction}
Magnetohydrodynamics (MHD in short) couples Navier-Stokes equations with Maxwell's equations to study the macroscopic behaviour of electrically conducting fluids such as plasmas and liquid metals (see ~\cite{GLBL} and ~\cite{ST}). Given a domain $\Omega \subset \R^3$ and a time interval $[0,T)$, the Cauchy problem for the \emph{incompressible, viscous, resistive MHD equations} consists of the equations
\begin{align}
& \partial_t u + (u \cdot \nabla) u - (b \cdot \nabla) b - \nu \Delta u + \nabla \Pi = 0, \label{Resistive MHD} \\
& \partial_t b + \curl (b \times u) + \mu \curl \curl b = 0, \label{Resistive MHD2} \\
& \dive u = \dive b = 0, \label{Resistive MHD3} \\
& u(\cdot,0) = u_0, \; b(\cdot,0) = b_0, \label{Resistive MHD4}
  \end{align}
where $b$ is the magnetic field, $u$ is the velocity field, $\Pi$ is the total pressure, $\nu > 0$ is the coefficient of viscosity, $\mu > 0$ is the coefficient of resistivity and the initial datas $u_0$ and $b_0$ are divergence-free. The Navier-Stokes equations are a special case of MHD where $b \equiv 0$. Furthermore, setting $\mu = \nu = 0$ one obtains the \emph{ideal MHD equations}, while in the case $\mu = 0 < \nu$, \eqref{Resistive MHD}--\eqref{Resistive MHD3} are called the \emph{non-resistive MHD equations}.

In this work we consider Leray-Hopf solutions of \eqref{Resistive MHD}--\eqref{Resistive MHD4} in a bounded domain $\Omega$ of $\R^3$ that has a $\mathscr{C}^{1,1}$ boundary $\Gamma$. As we want to incorporate Tokamaks and other laboratory plasma configurations, it is mandatory to consider multiply connected domains (see Assumptions \ref{Assumption on Omega}--\ref{Assumption 2 on Omega} for the exact conditions on $\Omega$). We use the standard \emph{no-slip} and \emph{perfect conductivity} boundary conditions
\begin{align}
& u|_\Gamma = 0, \label{Resistive MHD5} \\
& b \cdot n|_\Gamma = 0 \qquad \text{and} \qquad (\curl b) \times n|_\Gamma = 0, \label{Resistive MHD6}
\end{align}
(see \textsection \ref{Leray-Hopf solutions of viscous, resistive MHD equations and the inviscid limit} for precise definitions). 

The existence of Leray-Hopf solutions in smooth simply connected domains goes back to ~\cite{DL} and ~\cite{ST}, and in ~\cite{XXW}, existence is shown under the slip without friction conditions on $u$. The more complicated case of smooth multiply connected domains is covered in the doctoral dissertation ~\cite{Khe}. Since ~\cite{Khe} is not readily available, we present our version of the proof for $\mathscr{C}^{1,1}$ multiply connected domains in the Appendix. For local-in-time existence and uniqueness of strong solutions as well as weak solutions in suitable Besov spaces for ideal MHD see ~\cite{MY}, ~\cite{Schmidt} and ~\cite{Secchi}, and for the case of non-resistive MHD see ~\cite{CMRR}, ~\cite{FMRR}, ~\cite{FMRR2} and ~\cite{LTY}. For further references see ~\cite[p. 57]{GLBL}.

\vspace{0.3cm}
In ideal 3D MHD, smooth solutions conserve the \emph{total energy} $2^{-1} \int_\Omega (\abs{u(x,t)}^2 + \abs{b(x,t)}^2) dx$ and the \emph{cross helicity} $\int_\Omega u(x,t) \cdot b(x,t) \, dx$ in time. In simply connected domains the \emph{magnetic helicity}
\[\int_\Omega \psi(x,t) \cdot b(x,t) \, dx,\]
where $\psi$ is a vector potential of $b$ (that is, $\curl \psi = b$), is also conserved by smooth solutions and is independent of the choice of $\psi$.

Recently obtained numerical evidence points, however, towards \emph{anomalous energy dissipation}, that is, the rate of total energy dissipation in viscous, resistive MHD \emph{does not tend to zero} when $\mu,\nu \to 0$ (when the Reynolds number and magnetic Reynolds number tend to infinity); see ~\cite{DA}, ~\cite{LBMM} and ~\cite{MP}. Thus, if ideal MHD equations are to be a good model for magnetohydrodynamic turbulence at very high Reynolds number and magnetic Reynolds number, then the equations must possess (physically realistic) energy dissipative solutions. This is in analogy to the celebrated Onsager Conjecture on Euler equations (see ~\cite{BDLSV}, ~\cite{CET}, ~\cite{Eyi}, ~\cite{Ise16} and ~\cite{Onsager}). In ideal MHD, bounded non-vanishing weak solutions with compact support in time (thus violating total energy conservation) were found in ~\cite{BLFNL}, while non-vanishing smooth strict subsolutions with compact support in space-time were constructed in ~\cite{FL}.

\vspace{0.3cm}
In stark contrast to total energy, magnetic helicity has proved to be a very robust time invariant of ideal MHD. First, Caflisch, Klapper and Steele showed in ~\cite{CKS} that magnetic helicity is conserved whenever $u \in C([0,T];B^{\alpha_1}_{3,\infty}(\T^3,\R^3))$ and $b \in C([0,T];B^{\alpha_2}_{3,\infty}(\T^3,\R^3))$ with $\alpha_1 + 2 \alpha_2 > 0$, and next Kang and Lee showed magnetic helicity conservation for $u,b \in C_w([0,T];L^2(\T^3,\R^3)) \cap L^3(0,T;L^3(\T^3,\R^3))$ in ~\cite{KL}. In ~\cite{FL}, the authors extended conservation to subsolutions and weak limits of solutions in $L^3(0,T;L^3(\T^3,\R^3))$.

It is still open whether magnetic helicity is conserved if $u$ and $b$ belong to the physically natural energy space $L^\infty(0,T;L^2(\T^3,\R^3))$. However, a straightforward adaptation of our Theorem \ref{Simply connected Taylor theorem} to the torus implies that conservation occurs if $u,b \in L^\infty(0,T;L^2(\T^3,\R^3))$ are a weak ideal limit of Leray-Hopf solutions (see Definition \ref{Definition of weak ideal limit} and Corollary \ref{Corollary on ideal limit}), which is arguably the physically relevant case.

\vspace{0.3cm}
It has been conjectured in the physics literature that magnetic helicity is approximately conserved at very low resistivities (see ~\cite{Tay} where the conjecture was first formulated by Taylor). Mathematically, the conjecture says that magnetic helicity is conserved in the ideal limit $\mu,\nu \to 0$ (see ~\cite[p. 444]{CKS}). Taylor's conjecture has been successful in explaining magnetic structures in laboratory plasmas, e.g., in the prediction of the relaxed state of a reversed field pinch, and lies at the heart of Taylor relaxation theory (for reviews with numerous further references see ~\cite{BCP} and ~\cite{Tay2}).

In Theorem \ref{Simply connected Taylor theorem} we prove Taylor's conjecture under weak and natural assumptions. We consider arbitrary weak limits of Leray-Hopf solutions when $\mu_j,\nu_j \to 0$ (which exist, up to a subsequence, whenever the $L^2$ norms of the initial datas are uniformly bounded). In particular, \emph{we do not assume that the weak limits satisfy the ideal MHD equations}. Recall that
\[L^2_\sigma(\Omega,\R^3) \defeq \{v \in L^2(\Omega,\R^3) \colon \dive v = 0, \; v \cdot n|_\Gamma = 0\}.\]

\begin{defin} \label{Definition of weak ideal limit}
Given Leray-Hopf solutions $(u_j, b_j)$ of \eqref{Resistive MHD}--\eqref{Resistive MHD6} with $\mu_j,\nu_j > 0$ and initial datas $u_{j,0},b_{j,0} \in L^2_\sigma(\Omega,\R^3)$ suppose that $\mu_j,\nu_j \to 0$ and that $u_{j,0} \rightharpoonup u_0$ and $b_{j,0} \rightharpoonup b_0$ in $L^2_\sigma(\Omega,\R^3)$. Assume that $u_j \overset{*}{\rightharpoonup} u$ and $b_j \overset{*}{\rightharpoonup} b$ in $L^\infty(0,T;L^2_\sigma(\Omega,\R^3))$. We then say that $(u,b)$ is a \emph{weak ideal limit of} $(u_j,b_j)$.

If instead $\mu_j \to 0$ and $\nu_j = \nu > 0$ for every $j \in \N$, we say that $(u,b)$ is a \emph{weak non-resistive limit of} $(u_j,b_j)$.
\end{defin}

Taylor's conjecture concerns the case where magnetic helicity is \emph{gauge invariant} (i.e. independent of the choice of the vector potential of $b$), that is, simply connected domains. The following theorem proves Taylor's conjecture.

\begin{thm} \label{Simply connected Taylor theorem}
Suppose $\Omega$ is simply connected and $(u,b)$ is a weak ideal limit of Leray-Hopf solutions $(u_j,b_j)$ with $\mu_j, \nu_j \to 0$. Then $\int_\Omega \psi(x,t) \cdot b(x,t) \, dx$ is a.e. constant in $t$ for every vector potential $\psi \in L^\infty(0,T;W^{1,2}(\Omega,\R^3))$ of $b$.
\end{thm}

Although in Theorem \ref{Simply connected Taylor theorem} we do not assume that $u$ and $b$ satisfy the ideal MHD equations, we present a corollary on solutions of ideal MHD. If a solution $(u,b)$ lies in the energy space $L^\infty(0,T;L^2_\sigma(\Omega,\R^3))$, then we may choose representatives $u,b \in C_w([0,T);L^2_\sigma(\Omega,\R^3))$; this can be proved by slightly modifying \cite[Lemmas 2.1--2.2]{Gal}.

\begin{cor} \label{Corollary on ideal limit}
Suppose $\Omega$ is simply connected and $u,b \in C_w([0,T);L^2_\sigma(\Omega,\R^3))$ form a weak solution of ideal MHD. If $(u,b)$ is a weak ideal limit of Leray-Hopf solutions $(u_j,b_j)$, then $b$ conserves magnetic helicity in time.
\end{cor}

While simply connected domains (and especially the torus $\T^3$) allow a relatively neat mathematical treatment, we also cover multiply connected domains in order to incorporate plasma containers in typical laboratory settings. The topology of multiply connected domains leads, however, to mathematical complications starting with the very definition of magnetic helicity.

\vspace{0.3cm}
Consider an arbitrary weak ideal limit $(u,b)$ of Leray-Hopf solutions $(u_j,b_j)$. If the domain $\Omega$ is multiply connected, then $\int_\Omega \psi(x,t) \cdot b(x,t) \, dx$ depends on the choice of the vector potential $\psi$. The basic reason behind this gauge dependence is the fact that when $\Omega$ is multiply connected, the orthogonal complement of $\ker(\curl)$ in $L^2(\Omega,\R^3)$ is a \emph{strict} subspace of $L^2_\sigma(\Omega,\R^3)$ -- in other words, the set of harmonic \emph{Neumann vector fields}
\[L^2_H(\Omega,\R^3) \defeq \{v \in L^2_\sigma(\Omega,\R^3) \colon \curl v = 0\}\]
is non-empty. For a physical interpretation of $L^2_H(\Omega,\R^3)$ see e.g. ~\cite[pp. 428--430]{CDTG}.

We write
\begin{equation} \label{Decomposition of L2sigma}
L^2_\sigma(\Omega,\R^3) = L^2_\Sigma(\Omega,\R^3) \oplus L^2_H(\Omega,\R^3);
\end{equation}
a useful intrinsic characterisation of $L^2_\Sigma(\Omega,\R^3)$ was given in ~\cite{FT} (see Theorem \ref{Foias-Temam decomposition theorem}). For the purposes of this article, it is also illuminating to use a characterisation familiar from Hodge-Friedrichs-Morrey decomposition theory (see ~\cite{Morrey}),
\begin{equation} \label{L2 Sigma 1}
L^2_\Sigma(\Omega,\R^3) = \{\curl \psi \colon \psi \in W^{1,2}(\Omega,\R^3), \; \psi \times n|_\Gamma = 0\}.
\end{equation}
In fact, we will need slightly more refined versions of \eqref{L2 Sigma 1}, see Theorem \ref{Time-dependent Borchers-Sohr} and Remark \ref{Remark on another gauge}.

Bearing in mind \eqref{Decomposition of L2sigma}, we decompose $b$ uniquely as
\begin{equation} \label{Decomposition of time-dependent fields}
b = b_\Sigma + b_H \qquad (b_\Sigma \in L^\infty(0,T;L^2_\Sigma(\Omega,\R^3)) \text{ and } b_H \in L^\infty(0,T;L^2_H(\Omega,\R^3)))
\end{equation}
and use similar notation for every $b_j$. In multiply connected domains, we prove that $\int_\Omega \psi(x,t) \cdot b(x,t) \, dx$ is conserved for all vector potentials $\psi \in L^\infty(0,T;W^{1,2}(\Omega,\R^3))$ of $b$ if and only if the harmonic part $b_H = 0$. There exist, however, weak ideal limits $(u,b)$ of Leray-Hopf solutions with $b_H \neq 0$ (see Proposition \ref{Answer of first question} for both claims).

We are thus led to the following question in multiply connected domains:
\begin{equation} \label{Second question}
\text{Is $\int_\Omega \psi(x,t) \cdot b(x,t) \, dx$ conserved for some natural class of potentials $\psi$?}
\end{equation}
We give a positive answer to \eqref{Second question} in Corollary \ref{Taylor corollary}. First, in Theorem \ref{Taylor theorem} we compute the magnetic helicity dissipation rate for arbitrary Leray-Hopf solutions and arbitrary vector potentials. In \eqref{Value of magnetic helicity} we are able to compute the dissipation rate also for weak ideal limits and all their potentials. Corollary \ref{Taylor corollary} then gives a condition on potentials that is coherent with \eqref{L2 Sigma 1} and yields magnetic helicity conservation.

\vspace{0.3cm}
We use the decomposition in \eqref{Decomposition of time-dependent fields} in order to give a formula for the time evolution of magnetic helicity. The components $b_\Sigma$ and $b_H$ of $b$ behave in rather differing ways; in particular, \emph{$b_H$ is constant in time} (see Proposition \ref{Decomposition of the magnetic field}). Because of difficulties described in \textsection \ref{The decomposition of vector potentials}, we also need to decompose $\psi$ in order to take advantage of the different features of $b_\Sigma$ and $b_H$:
\[\psi = \psi^\Sigma + \psi^H \qquad (\curl \psi^\Sigma = b_\Sigma \text{ and } \curl \psi^H = b_H).\]
The decomposition $\psi = \psi^\Sigma + \psi^H$ is not unique, and a judicious choice of the components $\psi^\Sigma, \psi^H \in L^\infty(0,T;W^{1,2}(\Omega,\R^3))$ is a fundamental part of the proof of Theorem \ref{Taylor theorem}. In fact, we end up performing a further decomposition of $\psi^\Sigma$, and the whole decomposition of $\psi$ is described in \textsection \ref{Good vector potentials} and \textsection \ref{The decomposition of vector potentials}.

In order to state Theorem \ref{Taylor theorem} we already note below that given $\psi$, there exists a canonical choice of $\psi^H$, and we use it for all the vector potentials in this article. In particular, with this choice, $\partial_t b_H = 0$ implies that $\partial_t \psi^H = 0$.

\begin{defin}
Suppose that $v = v_\Sigma + v_H \in L^\infty(0,T;L^2_\sigma(\Omega,\R^3))$ and that $\psi \in L^\infty(0,T;W^{1,2}(\Omega,\R^3))$ satisfies $\curl \psi = v$. We denote by $\psi^H$ the  unique mapping in $L^\infty(0,T; W^{1,2}(\Omega,\R^3) \cap L^2_\Sigma(\Omega,\R^3))$ such that $\curl \psi^H = b_H$ (see Theorem \ref{Amrouche-Bernardi-Dauge-Girault lemma}), and we furthermore denote $\psi^\Sigma \defeq \psi - \psi^H$.
\end{defin}

We are now in a position to state our main theorem; the strategy of the proof is described in \textsection \ref{The decomposition of vector potentials}--\ref{Strategy of the proof}, and the details are presented in \textsection \ref{Reduction to good vector potentials}--\ref{Completion of the proof}.

\begin{thm} \label{Taylor theorem}
Suppose a domain $\Omega \subset \R^3$ satisfies Assumption \ref{Assumption on Omega}, and assume that $(u,b)$ is a weak ideal limit or weak resistive limit of Leray-Hopf solutions $(u_j,b_j)$, $j \in \N$. Then any vector potentials $\psi_j$ and $\psi_{j,0}$ of $b_j$ and $b_{j,0}$ satisfy
\begin{align*}
    \int_\Omega \psi_j(x,t) \cdot b_j(x,t) \, dx
&= \int_\Omega \psi_{j,0}(x) \cdot b_{j,0}(x) \, dx \\
&- 2 \mu_j \int_0^t \int_\Omega b_j(x,\tau) \cdot \curl b_j(x,\tau) \, dx \, d\tau \\
&- \int_\Gamma [\psi_j^\Sigma(x,t)-\psi^\Sigma_{j,0}(x)] \times n \cdot \psi^H_{j,0}(x) \, dx
\end{align*}
for all $j \in \N$ and $t \in [0,T)$. Furthermore,
\begin{equation} \label{Value of magnetic helicity}
\int_\Omega \psi(x,t) \cdot b(x,t) \, dx = \int_\Omega \psi_0(x) \cdot b_0(x) dx - \int_\Gamma [\psi^\Sigma(x,t)-\psi^\Sigma_0(x)] \times n \cdot \psi^H_0(x) \, dx
\end{equation}
for a.e. $t \in (0,T)$ and all vector potentials $\psi$ and $\psi_0$ of $b$ and $b_0$.
\end{thm}

Formula \eqref{Value of magnetic helicity} allows us to show magnetic helicity conservation for a large class of vector potentials. The class is specified in \eqref{Natural condition 2}, and its naturality is apparent from \eqref{L2 Sigma 1} and \eqref{Value of magnetic helicity}.

\begin{cor} \label{Taylor corollary}
Suppose the assumptions of Theorem \ref{Taylor theorem} hold. If
\begin{equation} \label{Natural condition 2}
\psi^\Sigma_j \times n|_\Gamma = \psi^\Sigma \times n|_\Gamma = 0 \quad{and} \quad \psi_{j,0}^\Sigma \times n|_\Gamma = \psi_0^\Sigma \times n|_\Gamma = 0,
\end{equation}
then
\begin{equation} \label{Magnetic helicity for good potentials}
\int_\Omega \psi(x,t) \cdot b(x,t) \, dx = \int_\Omega \psi_0(x) \cdot b_0(x) dx = \lim_{j \to \infty} \int_\Omega \psi_{j,0}(x) \cdot b_{j,0}(x) \, dx
\end{equation}
for a.e. $t \in (0,T)$. In particular, under condition \eqref{Natural condition 2}, the magnetic helicity of $b$ is independent of the choice of $\psi$.
\end{cor}

In \textsection \ref{A two-dimensional analogue} we prove a two-dimensional analogue of Theorem \ref{Taylor theorem}: in bounded, multiply connected Lipschitz domains, mean-square magnetic potential is conserved in the weak ideal limit. In 2D, there exists a canonical choice of potentials, and so we can follow the philosophy of ~\cite{FL} which is based on $\mathcal{H}^1$-$\operatorname{BMO}$ duality and compensated compactness. In fact, we also show that in multiply connected domains, all solutions of ideal MHD in the energy space conserve magnetic mean-square potential, extending a similar result on the torus $\T^2$ from ~\cite{FL}.

\vspace{0.3cm}
In three dimentions, when magnetic field lines are allowed to cross $\Gamma$, that is, the assumption $b \cdot n|_{\Gamma} = 0$ is dropped, magnetic helicity is no longer gauge invariant even for smooth solutions of ideal MHD in simply connected domains. In such a setting the so-called \emph{relative magnetic helicity}, defined in ~\cite{BF} and ~\cite{FA}, can be studied instead. We defer a treatment of relative magnetic helicity to a subsequent work.

\section{Background}

In this chapter we review tools and results needed in this article. We first fix our assumptions on the domain $\Omega$ in \textsection \ref{Assumptions on the domain} and recall basic material on boundary traces of Sobolev and $L^p$ functions in \textsection \ref{Traces of Sobolev functions}; \textsection \ref{Bochner spaces} reviews some standard results on time-dependent mappings in Bochner spaces, and in \textsection \ref{Leray-Hopf solutions of viscous, resistive MHD equations and the inviscid limit} we discuss Leray-Hopf solutions of viscous, resistive 3D MHD equations and the notion of inviscid, non-resistive limit.

\subsection{Assumptions on the domain} \label{Assumptions on the domain}
We start by fixing our assumptions on the domain $\Omega$, and we illustrate the assumptions in Figure \ref{Torus figures}. Our exposition follows ~\cite[pp. 835--836]{ABDG} (see also ~\cite{Tem}).

\begin{assumption} \label{Assumption on Omega}
The domain $\Omega \subset \R^3$ is bounded and its boundary $\Gamma$ is of class $\mathscr{C}^{1,1}$ and has a finite number of connected components denoted by $\Gamma_1,\ldots,\Gamma_K$.
\end{assumption}

Another assumption is introduced in order to produce a simply connected domain by making cuts into $\Omega$. The cuts will, however, only play an implicit role in this article.

\begin{assumption} \label{Assumption 2 on Omega}
There exist connected open surfaces $\Sigma_j$, $1 \le j \le N$, contained in $\Omega$ and satisfying the following conditions:
\renewcommand{\labelenumi}{(\roman{enumi})}
\begin{enumerate}

\item Each surface $\Sigma_j$ is an open subset of a smooth manifold $\mathscr{M}_j$.

\item The boundary of each $\Sigma_j$ is contained in $\partial \Omega$.

\item $\bar{\Sigma}_i \cap \bar{\Sigma}_j = \emptyset$ whenever $i \neq j$.

\item The open set $\dot{\Omega} \defeq \Omega \setminus \sup_{j=1}^N \Sigma_j$ is simply connected and pseudo-Lipschitz (see Definition \ref{Definition of pseudo-Lipschitz domains} below).
\end{enumerate}
The sets $\Sigma_j$ are called \emph{cuts}.
\end{assumption}

\begin{figure}[b]
    \center
    \subfloat[Projection of a torus $\Omega \subset \R^3$ into the $xy$-plane.]{\includegraphics[width=0.47\textwidth]{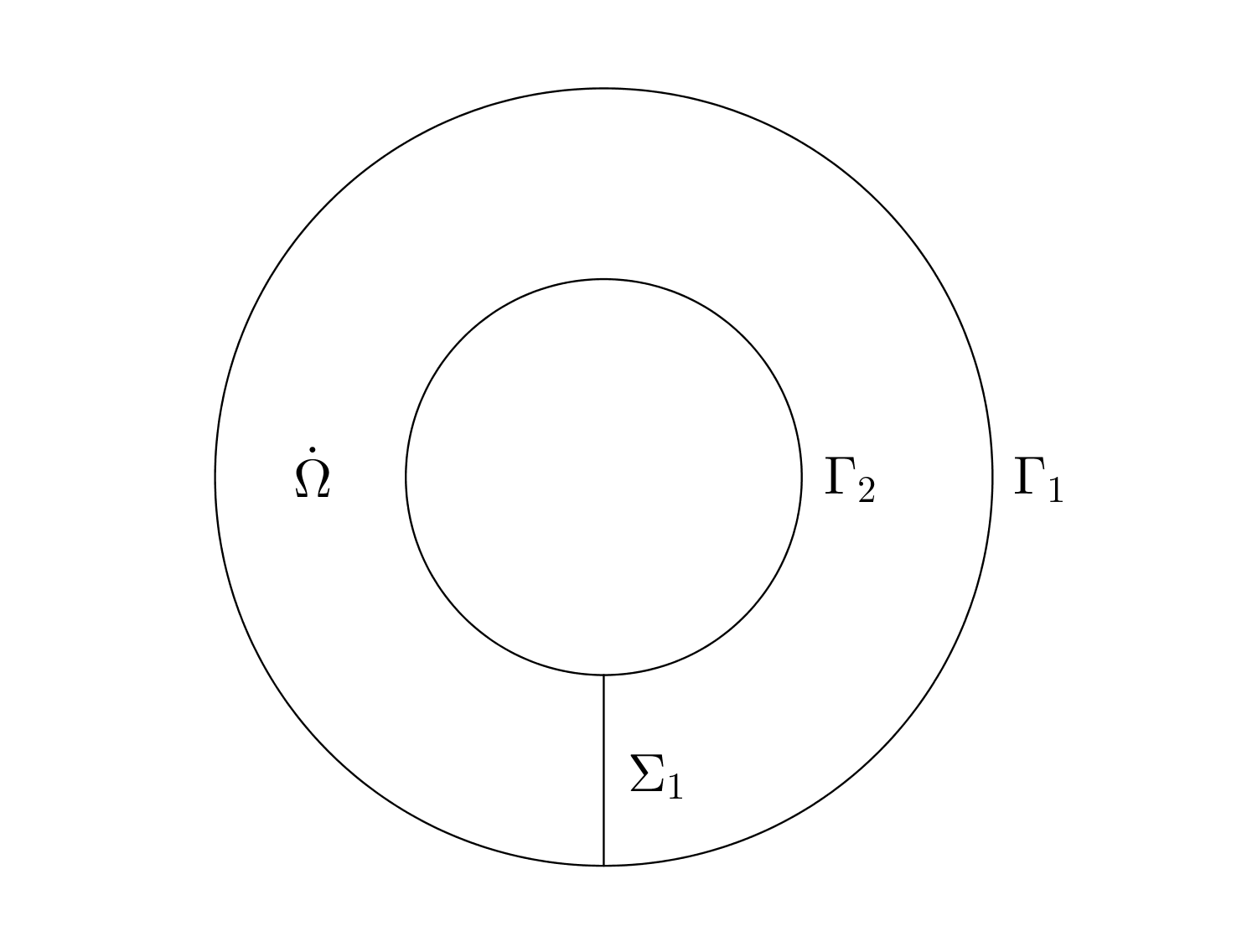}}
    \hfill
    \subfloat[Projection of a double torus into the $xy$-plane.]
{\includegraphics[width=0.53\textwidth]{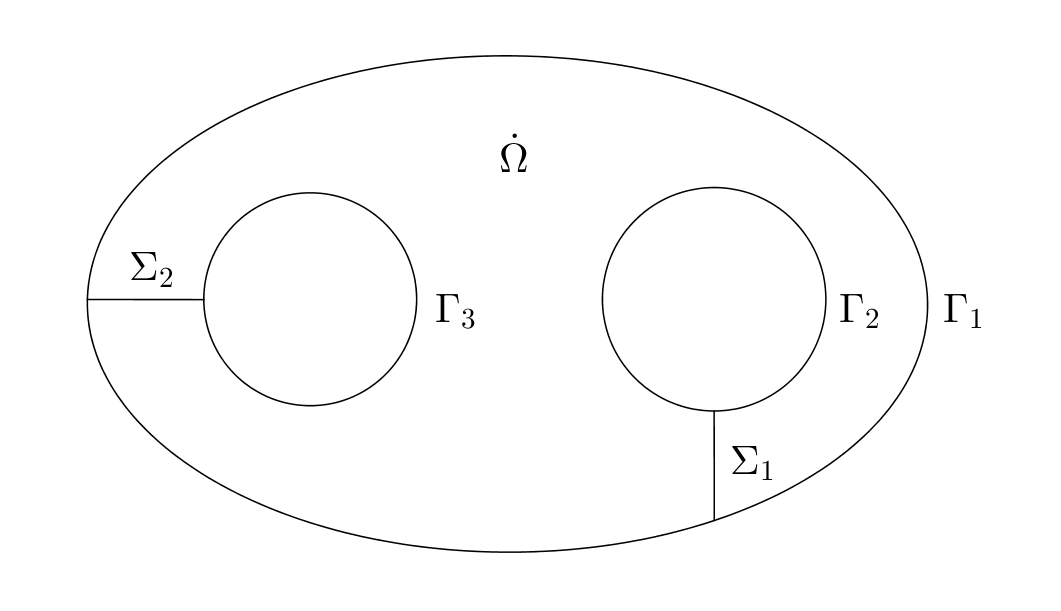}}
    \caption{}
    \label{Torus figures}
\end{figure}

The notion of a pseudo-Lipschitz domain is a generalization of a Lipschitz domain that allows the domain to locally lie on both sides of its boundary.

\begin{defin} \label{Definition of pseudo-Lipschitz domains}
A bounded domain $\Delta \subset \R^3$ is called \emph{pseudo-Lipschitz} if for every $x \in \partial \Delta$ there exists an integer $r(x) \in \{1,2\}$ and a radius $\rho_0 > 0$ such that whenever $0 < \rho < \rho_0$, the intersection $\Delta \cap B(x,\rho)$ has $r(x)$ connected components, each one with a Lipschitz boundary.
\end{defin}

Assumptions \ref{Assumption on Omega}--\ref{Assumption 2 on Omega} are standard in the study of fluid dynamics in multiply connected domains (see e.g. ~\cite{ABDG}, ~\cite{FT} and ~\cite{Tem}) and will remain in place for the rest of this article (except \textsection \ref{A two-dimensional analogue} where we discuss the two-dimensional setting). In particular, a solid torus clearly satisfies Assumptions \ref{Assumption on Omega}--\ref{Assumption 2 on Omega}.

\subsection{Traces of Sobolev functions} \label{Traces of Sobolev functions}
We recall results on boundary traces, normal traces and tangential traces and refer to ~\cite{Gal2}, ~\cite{GR} and ~\cite{Necas} for the proofs. In Theorems \ref{Trace theorem}--\ref{Tangential trace theorem} the assumption that $\Gamma$ is $\mathscr{C}^{1,1}$ can in fact be relaxed to $\Gamma$ being Lipschitz regular. The first trace theorem we present is a special case of ~\cite[Theorem II.4.1]{Gal2}.

\begin{thm} \label{Trace theorem}
Let $1 \le p < \infty$. Then there exists a unique, continuous linear map $\gamma \colon W^{1,p}(\Omega) \to L^p(\Gamma)$ such that for every $u \in C^\infty(\bar{\Omega})$ we have $\gamma(u) = u$ on $\Gamma$.
\end{thm}

We denote by $W^{1-1/p,p}(\Gamma)$ the subspace of $L^p(\Gamma)$ of functions for which
\[\|u\|_{W^{1-1/p,p}(\Gamma)} \defeq \|u\|_{L^p(\Omega)} + \left( \int_{\Gamma} \int_{\Gamma} \frac{|u(x)-u(y)|^p}{|x-y|^{1+p}} \, dS(x) \, dS(y) \right)^\frac{1}{p} < \infty.\]
The space $W^{1-1/p,p}(\Gamma)$ is dense in $L^p(\Gamma)$ and complete in the norm $\|\cdot\|_{W^{1-1/p,p}(\Gamma)}$. When $1 < p < \infty$, the trace operator $\gamma$ is a Banach space isomorphism from the quotient space $W^{1,p}(\Omega) / W_0^{1,p}(\Omega)$ onto $W^{1-1/p,p}(\Gamma)$ (see ~\cite[\textsection 2.5, Theorems 5.5 and 5.7]{Necas}):

\begin{thm} \label{Trace is an isomorphism}
Let $1 < p < \infty$. If $u \in W^{1,p}(\Omega)$, then $\gamma(u) \in W^{1-1/p,p}(\Gamma)$ and
\[\|\gamma(u)\|_{W^{1-1/p,p}(\Gamma)} \lesssim_{\Omega,p} \|u\|_{W^{1,p}(\Omega)}.\]
Conversely, given $w \in W^{1-1/p,p}(\Gamma)$ there exists $u \in W^{1,p}(\Omega)$ such that $\gamma(u) = w$ and $\|u\|_{W^{1,p}(\Omega)} \lesssim_{\Omega,p} \|\gamma(u)\|_{W^{1-1/p,p}(\Gamma)}$.
\end{thm}

For convenience we will denote the trace $\gamma(u)$ simply by $u$. Whenever $u \in C^\infty(\bar{\Omega},\R^3)$, the \emph{normal trace} $u \cdot n$ and the \emph{tangential trace} $u \times n$ are well-defined on the boundary $\Gamma$ and the generalized Gauss identity and Green's formula
\begin{equation} \label{Gauss identity}
\langle u \cdot n, \varphi \rangle_{\Gamma} = \int_\Omega u(x) \cdot \nabla \varphi(x) \, dx + \int_\Omega \varphi(x) \dive u(x) \, dx, \qquad \varphi \in W^{1,p'}(\Omega)
\end{equation}
\begin{equation} \label{Green's formula}
\langle u \times n, \psi \rangle_{\Gamma} = \int_\Omega \curl u(x) \cdot \psi(x) \, dx - \int_\Omega u(x) \cdot \curl \psi(x) \, dx, \qquad \psi \in W^{1,p'}(\Omega,\R^3)
\end{equation}
hold, where $\langle u \cdot n, \varphi \rangle_{\Gamma}$ and $\langle u \times n, \psi \rangle_{\Gamma}$ are standard surface integrals (but can also be interpreted in terms of $W^{-1/p,p}(\Gamma)$--$W^{1-1/p',p'}(\Gamma)$ duality).

Normal and tangential traces are extended to the function spaces defined next: when $1 < p < \infty$, $H^p(\dive,\Omega) \defeq \{v \in L^p(\Omega,\R^3) \colon \dive v \in L^p(\Omega)\}$ is endowed with the norm $\norm{v}_{H^p(\dive,\Omega)} \defeq (\norm{v}_{L^p(\Omega)}^p + \norm{\dive v}_{L^p(\Omega)}^p)^{1/p}$, while $H^p(\curl,\Omega) \defeq \{v \in L^p(\Omega,\R^3) \colon \curl v \in L^p(\Omega,\R^3)\}$ is given the norm $\norm{v}_{H^p(\curl,\Omega)} \defeq (\norm{v}_{L^p(\Omega)}^p + \norm{\curl v}_{L^p(\Omega)}^p)^{1/p}$.

\begin{thm} \label{Normal trace theorem}
Suppose $1 < p < \infty$. Then the normal trace has a unique bounded extension $u \mapsto u \cdot n \colon H^p(\dive,\Omega) \to W^{-1/p,p}(\Gamma)$ and the generalized Gauss identity \eqref{Gauss identity} holds.
\end{thm}

For a proof of Theorem \ref{Normal trace theorem} see ~\cite[Theorem III.2.2]{Gal2}. In a similar vein, a tangential trace is well-defined whenever $v \in L^p(\Omega,\R^3)$ and $\curl v \in L^p(\Omega,\R^3)$:

\begin{thm} \label{Tangential trace theorem}
Suppose $1 < p < \infty$. Then the tangential trace has a unique bounded extension $u \mapsto u \times n \colon H_p(\curl;\Omega) \to W^{-1/p,p}(\Gamma,\R^3)$ and the generalized Green's formula \eqref{Green's formula} holds.
\end{thm}

Finally we mention a characterisation of $W^{1,2}(\Omega,\R^3)$ by Foias and Temam (see e.g. ~\cite[Corollary 3.7]{GR}). Here Lipschitz continuity of $\Gamma$ would not be sufficient (see ~\cite[p. 832]{ABDG}).

\begin{thm} \label{Foias-Temam theorem}
We have $W^{1,2}(\Omega,\R^3) = \{v \in L^2(\Omega,\R^3) \colon \dive v \in L^2(\Omega), \curl v \in L^2(\Omega,\R^3), v \cdot n \in W^{1/2,2}(\Gamma)\}$ and
\[\|v\|_{W^{1,2}(\Omega)} \lesssim_\Omega \|v\|_{L^2(\Omega)} + \|\dive v\|_{L^2(\Omega)} + \|\curl v\|_{L^2(\Omega)} + \|v \cdot n\|_{W^{1/2,2}(\Gamma)}\]
for all $v \in W^{1,2}(\Omega,\R^3)$.
\end{thm}

\subsection{Bochner spaces} \label{Bochner spaces}
We recall some basic facts on time-dependent mappings in Bochner spaces in a generality needed in this article. We do not discuss the definitions of Bochner measurability and Bochner integrability but refer to ~\cite{HVNVW} for a thorough introduction to Bochner spaces and to ~\cite{Rou} for a shorter one with an emphasis on applications in PDE's.

Whenever $1 \le p < \infty$ and $X$ is a Banach space, the Bochner space $L^p(0,T;X)$ consists of (classes with respect to equality a.e. $t \in (0,T)$ of) Bochner integrable functions $v \colon (0,T) \to X$ satisfying $\int_0^T \|v(\cdot,t)\|_X^p dt < \infty$. For  $L^\infty(0,T;X)$ the corresponding condition is $\norm{\norm{v(\cdot,t)}_X}_{L^\infty(0,T)} < \infty$. If $1 \le p < \infty$ and $X^*$ is separable, then $(L^p(0,T;X))^* = L^{p'}(0,T;X^*)$ with the duality pairing given by
\[\langle f, v \rangle_{L^{p'}(0,T;X^*)-L^p(0,T;X)} \defeq \int_0^T \langle f(\cdot,t), v(\cdot,t) \rangle_{X^*-X} dt\]
(see ~\cite[Corollary 1.3.22]{HVNVW}). Furthermore, then $L^p(0,T;X)$ is separable (see ~\cite[Proposition 1.2.29]{HVNVW}) and thus every bounded sequence in $L^{p'}(0,T;X^*)$ has a weak-$*$ convergent subsequence. We also denote by $C_w([0,T); X)$ the set of mappings $v \colon [0,T) \to X$ defined at every $t \in [0,T)$ and satisfying $t_j \to t$ in $[0,T)$ $\Rightarrow$ $v(\cdot,t_j) \rightharpoonup v(\cdot,t)$ in $X$.

Whenever $f \in L^1(0,T;X)$, $0 < \delta < T/2$ and $\theta \in C_c^\infty(\R)$ with $\supp(\theta) \subset (-\delta,\delta)$, we define $f * \theta \in C^\infty(\delta,T-\delta; X)$ by $f * \theta(\cdot,t) \defeq \int_0^T \theta(t-s) f(\cdot,s) \, ds \in X$. We record a variant of Young's convolution inequality.

\begin{lem} \label{Lemma on time integrability of mollified functions}
Suppose $p,q,r \in [1,\infty]$ with $1/p+1/q=1+1/r$ and $1 \le s < \infty$. If $v \in L^p(0,T; L^s(\Omega))$ and $\theta \in C_c^\infty(\R)$ with $\supp(\theta) \subset (-\delta,\delta)$, then
\[\norm{v * \theta}_{L^r(\delta,T-\delta; L^s(\Omega))} \le \norm{v}_{L^p(0,T; L^s(\Omega))} \norm{\theta}_{L^q(-\delta,\delta)}.\]
\end{lem}

\begin{proof}
By Minkowski's integral inequality and Young's convolution inequality,
\[\begin{array}{lcl}
& & \displaystyle \int_\delta^{T-\delta} \left( \int_\Omega |v * \theta(x,t)|^s dx \right)^\frac{r}{s} dt \\
&=& \displaystyle \int_\delta^{T-\delta} \left( \int_\Omega \left| \int_0^T v(x,\tau) \theta(t-\tau) \, d\tau \right|^s dx \right)^\frac{r}{s} dt \\
&\le& \displaystyle \int_\delta^{T-\delta} \left( \int_0^T \left( \int_\Omega \abs{v(x,\tau)}^s dx \right)^\frac{1}{s} \abs{\theta(t-\tau)} d\tau \right)^r dt \\
&=& \displaystyle \int_\delta^{T-\delta} (\norm{x \mapsto v(x,\cdot)}_{L^s(\Omega)} * \abs{\theta}(t))^r dt \\
&\le& \displaystyle \norm{v}_{L^p(0,T; L^s(\Omega))}^r \norm{\theta}_{L^q(-\delta,\delta)}^r.
\end{array}\]
\end{proof}

We fix, for the rest of this article, an even mollifier $\chi \in C_c^\infty(\R)$ with $\supp(\chi) \subset (-1,1)$ and $\int_{-1}^1 \chi(t) \, dt = 1$. We denote $\chi^\delta(t) \defeq \delta^{-1} \chi(t/\delta)$ for all $\delta > 0$ and $t \in \R$. For every $f \in L^1(0,T;L^1(\Omega))$ we denote $f_\delta \defeq f * \chi^\delta$. For a proof of the following mollifier approximation lemma see ~\cite[Proposition 1.2.32]{HVNVW}.

\begin{lem} \label{Mollifier approximation lemma}
Let $0 < \epsilon < T/2$ and suppose $1 \le p,q < \infty$ and $f \in L^p(0,T;L^q(\Omega))$. Then $\|f_\delta - f\|_{L^p(\epsilon,T-\epsilon;L^q(\Omega))} \to 0$ as $\delta \to 0$.
\end{lem}

The following interpolation inequalities will also be useful to us.

\begin{lem} \label{Interpolation lemma}
For every $v,w \in L^\infty(0,T; L^2(\Omega,\R^3)) \cap L^2(0,T;W^{1,2}(\Omega,\R^3))$ we have
\begin{align*}
& \norm{v}_{L^4(0,T;L^3(\Omega))} \lesssim_\Omega \norm{v}_{L^2(0,T;W^{1,2}(\Omega)}^{1/2} \norm{v}_{L^\infty(0,T;L^2(\Omega))}^{1/2}, \\
& \norm{v \otimes w}_{L^1(0,T;W^{1,3/2}(\Omega))} \lesssim_\Omega \|v\|_{L^2(0,T;W^{1,2}(\Omega))} \|w\|_{L^2(0,T;W^{1,2}(\Omega))}, \\
& \norm{v \otimes w}_{L^{4/3}(0,T;L^2(\Omega))} \lesssim_\Omega \|v\|_{L^2(0,T;W^{1,2}(\Omega))}^{3/4} \|v\|_{L^\infty(0,T;L^2(\Omega))}^{1/4} \\
&\cdot \|w\|_{L^2(0,T;W^{1,2}(\Omega))}^{3/4} \|w\|_{L^\infty(0,T;L^2(\Omega))}^{1/4},
\end{align*}
where $v \otimes w \defeq [v_i w_j]_{i,j=1}^3$ is the tensor product of $v$ and $w$.
\end{lem}

\begin{proof}
The first inequality is a standard interpolation and can be found e.g. at ~\cite[p. 74]{RRS} (up to a use of the Sobolev embedding $W^{1,2}(\Omega) \subset L^6(\Omega)$). For the second one note that at a.e. $t \in (0,T)$, H\"{o}lder's inequality and the Sobolev embedding $W^{1,2}(\Omega) \hookrightarrow L^6(\Omega)$ yield 
\[\norm{\abs{v} \abs{\nabla w}}_{L^{3/2}(\Omega)}
\le \norm{v}_{L^6(\Omega)} \norm{\nabla w}_{L^2(\Omega)}
\lesssim_\Omega \norm{v}_{W^{1,2}(\Omega)} \norm{w}_{W^{1,2}(\Omega)}.\]
A similar inequality holds for $\abs{w} \abs{\nabla w}$ and $\abs{v} \abs{w}$, and one then uses the Cauchy-Schwarz inequality on time integrals to finish the proof. Similar reasoning is used to prove the third inequality of the lemma.
\end{proof}

We also recall the Aubin-Lions Lemma which we formulate in a form that suffices for the purposes of this article (see ~\cite[Lemma 7.7]{Rou}).

\begin{lem} \label{Aubin-Lions lemma}
Let $X$, $Y$ and $Z$ be reflexive Banach spaces such that $X$ embeds compactly into $Y$ and $Y$ embeds into $Z$. Suppose $1 < p < \infty$ and $1 \le q \le \infty$. Then $\{u \in L^p(0,T;X) \colon \partial_t u \in L^q(0,T;Z)\}$ embeds compactly into $L^p(0,T;Y)$.
\end{lem}

\subsection{Leray-Hopf solutions of viscous, resistive MHD equations and the inviscid, non-resistive limit} \label{Leray-Hopf solutions of viscous, resistive MHD equations and the inviscid limit}
We recall the definition and present an existence theorem on Leray-Hopf solutions of viscous, resistive 3D MHD equations. When $1 < p < \infty$, we denote the relevant function spaces by
\begin{align*}
& C_{c,\sigma}^\infty(\Omega,\R^3) \defeq \{\varphi \in C_c^\infty(\Omega,\R^3) \colon \dive \varphi = 0\}, \\
& L^p_\sigma(\Omega,\R^3) \defeq \overline{C_{c,\sigma}^\infty(\Omega,\R^3)}^{L^p(\Omega,\R^3)} = \{v \in L^p(\Omega,\R^3) \colon \dive v = 0, \; v \cdot n|_{\Gamma} = 0\}, \\
& W^{1,p}_{0,\sigma}(\Omega,\R^3) \defeq \overline{C_{c,\sigma}^\infty(\Omega,\R^3)}^{W_0^{1,p}(\Omega,\R^3)} = \{v \in W_0^{1,p}(\Omega,\R^3) \colon \dive v = 0\}, \\
& W^{1,p}_\sigma(\Omega,\R^3) \defeq \{v \in W^{1,p}(\Omega,\R^3) \colon \dive v = 0, \; v \cdot n|_{\Gamma} = 0\}.
\end{align*}
(for the two identities see e.g. ~\cite[Theorems III.2.3 and III.4.1]{Gal2}).
Leray-Hopf solutions of MHD are defined by the following standard variational formulation.

\begin{defin} \label{Leray-Hopf solutions}
Let $u_0,b_0 \in L^2_\sigma(\Omega,\R^3)$. Suppose that $u \in C_w([0,T);L^2_\sigma(\Omega,\R^3)) \cap L^2(0,T; W^{1,2}_{0,\sigma}(\Omega))$ and $b \in C_w([0,T);L^2_\sigma(\Omega,\R^3)) \cap L^2(0,T; W^{1,2}_\sigma(\Omega))$ satisfy $\partial_t u \in L^1(0,T;(W^{1,2}_{0,\sigma}(\Omega,\R^3))^*)$ and $\partial_t b \in L^1(0,T;(W^{1,2}_\sigma(\Omega,\R^3))^*)$, and that
\begin{align}
& \langle \partial_t u, \varphi \rangle + \int_{\Omega} (u \cdot \nabla u - b \cdot \nabla b) \cdot \varphi + \nu \int_\Omega \nabla u : \nabla \varphi = 0, \label{Resistive MHD weak definition 1} \\
& \langle \partial_t b, \theta \rangle + \int_{\Omega} b \times u \cdot \curl \theta + \mu \int_\Omega \curl b \cdot \curl \theta = 0 \label{Resistive MHD weak definition 2a}
\end{align}
hold at a.e. $t \in [0,T)$ and every $\varphi \in W^{1,2}_{0,\sigma}(\Omega,\R^3)$ and $\theta \in W^{1,2}_\sigma(\Omega,\R^3)$. Suppose furthermore that $u(\cdot,0) = u_0$ and $b(\cdot,0) = b_0$ and that $u$ and $b$ satisfy the \emph{energy inequality}
\[\begin{array}{lcl}
& & \displaystyle \frac{1}{2} \int_{\Omega} (\abs{u(x,t)}^2 + \abs{b(x,t)}^2) \, dx \\
&+& \displaystyle \int_0^t \int_{\Omega} (\nu \abs{\nabla u(x,\tau)}^2 + \mu \abs{\curl b(x,\tau)}^2) \, dx \, d\tau \\
&\le& \displaystyle \frac{1}{2} \int_{\Omega} (\abs{u_0(x)}^2 + \abs{b_0(x)}^2) \, dx
\end{array}\]
for all $t \in (0,T)$. Then $(u,b)$ is called a \emph{Leray-Hopf solution} of \eqref{Resistive MHD}--\eqref{Resistive MHD6}.
\end{defin}

Note that \eqref{Resistive MHD weak definition 2a} captures in a weak sense the condition $(\curl b) \times n|_\Gamma = 0$. Also note that \eqref{Resistive MHD weak definition 2a} and the condition $b(\cdot,0) = b_0$ imply
\begin{equation} \label{Resistive MHD weak definition 2}
\int_0^T \partial_t \eta \int_\Omega b \cdot \theta - \int_0^T \eta \int_{\Omega} b \times u \cdot \curl \theta - \mu \int_0^t \eta \int_\Omega \curl b \cdot \curl \theta + \eta(0) \int_\Omega b_0 \cdot \theta = 0
\end{equation}
for all $\eta \in C_c^\infty([0,T))$ and $\theta \in W^{1,2}_\sigma(\Omega,\R^3)$. As mentioned in the introduction, we present a proof of the following theorem in the Appendix.

\begin{thm} \label{Theorem on Leray-Hopf solutions}
Let $u_0,b_0 \in L^2_\sigma(\Omega;\R^3)$. Then there exists a Leray-Hopf solution $(u,b)$ of \eqref{Resistive MHD}--\eqref{Resistive MHD6}.
\end{thm}

Theorems \ref{Simply connected Taylor theorem} and \ref{Taylor theorem} do not assume that the inviscid, non-resistive (i.e. ideal) limit, defined below, holds. However, we mention the notion for completeness and also because it falls under the scope of Corollary \ref{Corollary on ideal limit}. It is a fundamental open problem under what conditions the inviscid, non-resistive limit holds in 3D MHD (see ~\cite{DLe}, ~\cite{WWLW}, ~\cite{Wu}, ~\cite{WW}, ~\cite{XXW} and ~\cite{Zhang} for partial results).

\begin{defin} \label{Definition of inviscid limit}
Suppose viscosities $\nu_j > 0$ and resistivities $\mu_j > 0$ satisfy $\nu_j, \mu_j \to 0$ and that divergence-free initial datas $u_{j,0} \to u_0$ and $b_{j,0} \to b_0$ in $L^2_\sigma(\Omega,\R^3)$. Assume that for every $j \in \N$, $(u_j,b_j)$ is a Leray-Hopf solution of \eqref{Resistive MHD}--\eqref{Resistive MHD6} and that $u,b \in L^\infty(0,T;L^2_\sigma(\Omega,\R^3))$ form a solution of \eqref{Resistive MHD}--\eqref{Resistive MHD4} with $\mu = \nu = 0$. We say that $(u,b)$ is the \emph{inviscid, non-resistive limit} or \emph{ideal limit} of $(u_j,b_j)$ (in the energy norm) if $\|u_j-u\|_{L^\infty(0,T;L^2(\Omega))} \to 0$ and $\|b_j - b\|_{L^\infty(0,T;L^2(\Omega))} \to 0$. We then also say that \emph{the inviscid, non-resistive limit holds for $(u_j,b_j)$ and $(u,b)$}.
\end{defin}

\section{Vector potentials and gauge dependence of magnetic helicity} \label{Vector potentials and the definition of magnetic helicity}
The aim of this section is to discuss the notion of magnetic helicity in multiply connected domains and to recall the existence of vector potentials satisfying the assumptions of Corollary \ref{Taylor corollary}.

\subsection{Magnetic helicity in multiply connected domains}
We first recall the Helmholtz-Weyl decomposition of $L^2(\Omega,\R^3)$ which is, in fact, valid in every domain of $\R^n$ for all $n \ge 2$ (see ~\cite[Theorem III.1.1]{Gal}).
\begin{thm} \label{Helmholtz-Weyl decomposition}
$L^2(\Omega) = L^2_\sigma(\Omega,\R^3) \oplus \nabla W^{1,2}(\Omega,\R^3)$.
\end{thm}

In ~\cite{FT}, Foias and Temam performed a further direct decomposition of $L^2_\sigma(\Omega,\R^3)$ into a part that has a vanishing flux across the cuts and an harmonic part (see ~\cite[Proposition 1.1]{FT} or ~\cite[Appendix I, Lemma 1.4]{Tem}). We present the decomposition of Foias and Temam in the notation of ~\cite{YG}.

\begin{thm} \label{Foias-Temam decomposition theorem}
$L^2_\sigma(\Omega,\R^3) = L^2_\Sigma(\Omega,\R^3) \oplus L^2_H(\Omega,\R^3)$, where
\begin{align*}
        L^2_\Sigma(\Omega,\R^3)
&\defeq \left\{ v \in L^2_\sigma(\Omega,\R^3) \colon \int_{\Sigma_i} v(x) \cdot n(x) \, dS(x) = 0 \text{ for } i = 1,\ldots,N \right\}, \\
        L^2_H(\Omega,\R^3)
&\defeq \{v \in L^2_\sigma(\Omega,\R^3) \colon \curl v = 0\}.
\end{align*}
\end{thm}

By Theorem \ref{Foias-Temam theorem}, $L^2_H(\Omega,\R^3) \subset W^{1,2}_\sigma(\Omega,\R^3)$.

\begin{defin}
We denote the projections onto $L^2_\Sigma(\Omega,\R^3)$ and $L^2_H(\Omega,\R^3)$ by $P_\Sigma \colon L^2_\sigma(\Omega,\R^3) \to L^2_\Sigma(\Omega,\R^3)$ and $P_H \colon L^2_\sigma(\Omega,\R^3) \to L^2_H(\Omega,\R^3)$. For every $v \in L^2_\sigma(\Omega,\R^3)$ we denote $v_\Sigma \defeq P_\Sigma v$ and $v_H \defeq P_H v$.
\end{defin}
The vector space $L^2_H(\Omega,\R^3)$ is $N$-dimensional. For a characterisation of an orthonormal basis $\{h_1,\ldots,h_N\}$ of $L^2_H(\Omega,\R^3)$ see ~\cite[Appendix I, Lemma 1.3]{Tem} or ~\cite[Proposition 3.14]{ABDG}. Theorems \ref{Helmholtz-Weyl decomposition} and \ref{Foias-Temam decomposition theorem} yield the decomposition
\begin{equation} \label{Variant of Helmholtz decomposition}
L^2(\Omega,\R^3) = L^2_\Sigma(\Omega,\R^3) \oplus \ker(\curl).
\end{equation}
Furthermore, $L^2_\sigma(\Omega,\R^3) \subset \{\curl \psi \colon \psi \in W^{1,2}(\Omega,\R^3)\}$ (see ~\cite[Appendix I, Proposition 1.3]{Tem}). We record the following simple observation.

\begin{prop} \label{Proposition on dependence on the vector potential}
Suppose $b \in L^2_\sigma(\Omega,\R^3)$. Then the value $\int_\Omega \psi(x) \cdot b(x) \, dx$ is independent of the solution $\psi \in W^{1,2}(\Omega,\R^3)$ of $\curl \psi = b$ if and only if $b \in L^2_\Sigma(\Omega,\R^3)$.
\end{prop}

\begin{proof}
If $b \in L^2_\sigma(\Omega,\R^3)$ and $\int_\Omega \phi(x) \cdot b(x) \, dx = 0$ for all $\phi \in W^{1,2}(\Omega,\R^3)$ with $\curl \phi = 0$, then in particular $\int_\Omega b(x) \cdot h_i(x) \, dx = 0$ for all $i \in \{1,\ldots,N\}$, giving $b \in L^2_\Sigma(\Omega,\R^3)$. The converse follows immediately from \eqref{Variant of Helmholtz decomposition}.
\end{proof}

Consequently, magnetic helicity is independent of the vector potential for every $b(\cdot,t) \in L^2_\sigma(\Omega,\R^3)$ precisely when $L^2_H(\Omega,\R^3) = \{0\}$. In Proposition \ref{Answer of first question} this helps us to characterise, in multiply connected domains, those magnetic fields whose magnetic helicity is conserved for every vector potential.

\begin{prop} \label{Answer of first question}
Suppose the mappings $u_j,b_j,u_{j,0},b_{j,0},u,b,u_0,b_0$ satisfy the assumptions of Theorem \ref{Taylor theorem}. Then the following conditions are equivalent.

\renewcommand{\labelenumi}{(\roman{enumi})}
\begin{enumerate}

\item $\int_\Omega \psi(x,t) \cdot b(x,t) \, dx$ is a.e. constant for every $\psi \in L^\infty(0,T;W^{1,2}(\Omega,\R^3))$ with $\curl \psi = b$.

\item $b_H = 0$.

\item $b_{0,H} = 0$.
\end{enumerate}
If $\Omega$ is multiply connected, there exist $u_j,b_j,u_{j,0},b_{j,0},u,b,u_0,b_0$ such that (i)--(iii) are not satisfied.
\end{prop}

\begin{proof}

The equivalence (i) $\Leftrightarrow$ (ii) is a direct corollary of Theorem \ref{Taylor theorem} and Proposition \ref{Proposition on dependence on the vector potential}, and the equivalence (ii) $\Leftrightarrow$ (iii) follows immediately from Lemma \ref{Decomposition of the magnetic field}. The last claim follows by combining Lemma \ref{Decomposition of the magnetic field} and Theorem \ref{Theorem on Leray-Hopf solutions}.
\end{proof}

Proposition \ref{Answer of first question} indicates that in multiply connected domains, magnetic helicity conservation can only hold in the weak ideal limit if some restrictions are imposed on the vector potential.

\subsection{Good vector potentials} \label{Good vector potentials}
As stated in Corollary \ref{Taylor corollary}, a condition that allows magnetic helicity conservation in multiply connected domains is given by
\begin{equation} \label{Natural condition}
\psi^\Sigma \times n|_\Gamma = 0 \qquad \text{and} \qquad \psi^\Sigma_0 \times n|_\Gamma = 0.
\end{equation}
We will, in fact, obtain Theorem \ref{Taylor theorem} as a consequence of the fact that \eqref{Natural condition} leads to magnetic helicity conservation. For more information on condition \eqref{Natural condition} see e.g. ~\cite{ABDG}, ~\cite{KY} and ~\cite{YG}.

Our next aim is to specify vector potentials that satisfy \eqref{Natural condition}. For the $L^2_\Sigma$ part of the magnetic field we essentially use vector potentials found by Borchers and Sohr in ~\cite[Corollary 2.2]{BS}. The boundary condition $\langle \partial_n(\dive \psi), 1 \rangle_{\Gamma_i} = 0$, added by Amrouche, Bernardi, Dauge and Girault in ~\cite{ABDG}, ensures uniqueness. Theorem \ref{Borchers-Sohr lemma} follows from ~\cite[Corollary 3.19]{ABDG} and ~\cite[Theorem 3.20]{ABDG}.

\begin{thm} \label{Borchers-Sohr lemma}
For every $v \in L^2_\Sigma(\Omega,\R^3)$ there exists a unique $T_\Sigma v  \defeq \Psi^\Sigma\in W^{1,2}_0(\Omega,\R^3)$ such that
\[\curl \Psi^\Sigma = v, \qquad \dive (\Delta \Psi^\Sigma) = 0, \qquad \langle \partial_n(\dive \Psi^\Sigma), 1 \rangle_{\Gamma_i} = 0 \quad (i = 1,\ldots,K).\]
Furthermore, $T_\Sigma \colon L^2_\Sigma(\Omega,\R^3) \to W_0^{1,2}(\Omega,\R^3)$ is linear and bounded.
\end{thm}

For the space $L^2_H(\Omega,\R^3)$ a natural choice of potentials is a special case of ~\cite[Theorem 3.12]{ABDG} and ~\cite[Corollary 3.16]{ABDG}:

\begin{thm} \label{Amrouche-Bernardi-Dauge-Girault lemma}
For every $v \in L^2_H(\Omega,\R^3)$ there exists a unique $T_H v \defeq \psi^H \in W^{1,2}(\Omega,\R^3) \cap L^2_\Sigma(\Omega,\R^3)$ such that $\curl (\psi^H) = v$. Furthermore, $T_H \colon L^2_H(\Omega,\R^3) \to W^{1,2}(\Omega,\R^3)$ is linear and bounded.
\end{thm}

We use Theorems \ref{Borchers-Sohr lemma} and \ref{Amrouche-Bernardi-Dauge-Girault lemma} to record an existence theorem about vector potentials satisfying \eqref{Natural condition}.

\begin{cor} \label{Time-dependent Borchers-Sohr}
For every $v \in L^\infty(0,T;L^2_\sigma(\Omega,\R^3))$, the mappings $\Psi^\Sigma(x,t) \defeq T_\Sigma v_\Sigma(x,t)$ and $\psi^H(x,t) \defeq T_H v_H(x,t)$ belong to $L^\infty(0,T;W^{1,2}(\Omega,\R^3))$ and satisfy $\curl \Psi^\Sigma = v_\Sigma$, $\curl \psi^H = v_H$ and \eqref{Natural condition}.
\end{cor}

Given $v \in L^\infty(0,T;L^2_\sigma(\Omega,\R^3))$, the time-dependent mappings $T_\Sigma v_\Sigma$ and $T_H v_H$ are strongly measurable, which follows from the fact that $T_\Sigma \circ P_\Sigma$ and $T_H \circ P_H$ are bounded linear operators from $L^2_\sigma(\Omega,\R^3)$ into $W^{1,2}(\Omega,\R^3)$.

\begin{rem} \label{Remark on another gauge}
Another choice of vector potentials that satisfies \eqref{Natural condition} (a special case of the \emph{Coulomb gauge}) is given in ~\cite[Theorem 3.17]{ABDG} and ~\cite[Corollary 3.19]{ABDG}: for every $v \in L^2_\Sigma(\Omega,\R^3)$ there exists a unique $\phi \in W^{1,2}(\Omega,\R^3)$ with $\curl \phi = v$, $\dive \phi = 0$, $\phi \times n = 0$ on $\Gamma$ and $\langle \phi \cdot n, 1 \rangle_{\Gamma_i} = 0$ for all $i \in \{1,\ldots,N\}$ - the condition $\psi^\Sigma \cdot n|_\Gamma = 0$ is thus traded for $\dive \phi = 0$.
\end{rem}

We will also need a scalar potential for time-dependent curl-free $L^p$ vector fields in simply connected domains.

\begin{lem} \label{Helmholtz-Hodge in domains}
Let $1 \le p \le \infty$ and $1 < q < \infty$, and suppose $\Omega' \subset \Omega$ is a simply connected domain with smooth boundary. If $v \in L^p(0,T; L^q(\Omega',\R^3))$ is curl-free, then there exists a unique $g \in L^p(0,T; W^{1,q}(\Omega'))$ such that $v = \nabla g$ and $\int_{\Omega'} g(x,t) \, dx = 0$. Furthermore, $\norm{g(\cdot,t)}_{W^{1,q}(\Omega')} \lesssim_{\Omega',q} \norm{v(\cdot,t)}_{L^q(\Omega')}$ for a.e. $t \in (0,T)$.
\end{lem}

\begin{proof}
For the existence of $g(\cdot,t)$ a.e. $t \in (0,T)$ see e.g. ~\cite[Lemma III.1.1]{Gal2}. Given $t \in (0,T)$ the uniqueness of $g(\cdot,t)$ is clear and the norm estimate follows from the Poincar\'{e} inequality. Since $v(\cdot,t) \mapsto g(\cdot,t)$ is linear, $g$ is strongly measurable.
\end{proof}

\section{The proof of Theorems \ref{Simply connected Taylor theorem} and \ref{Taylor theorem}} \label{The proof of Taylor theorem}
Theorem \ref{Taylor theorem} is proved in this section, and Theorem \ref{Simply connected Taylor theorem} is obtained as a special case. We begin by motivating our decomposition of vector potentials.

\subsection{The decomposition of vector potentials} \label{The decomposition of vector potentials}
Given a weak ideal or non-resistive limit $(u,b)$, our aim is to compute $\int_\Omega \psi(x,t) \cdot b(x,t) \, dx$ at a.e. $t \in (0,T)$ for every vector potential $\psi \in L^\infty(0,T;W^{1,2}(\Omega,\R^3))$ of $b$. However, we do not assume that $(u,b)$ satisfies the ideal MHD equations and so no neat formula for $\int_\Omega \psi(x,t) \cdot b(x,t) \, dx$ is readily available. We therefore wish to relate $\int_\Omega \psi(x,t) \cdot b(x,t) \, dx$ to $\int_\Omega \psi_j(x,t) \cdot b_j(x,t) \, dx$ and compute the latter for every $\psi_j$.

A natural idea for computing $\int_\Omega \psi_j(x,t) \cdot b_j(x,t) \, dx$ (which works without major complications in simply connected domains) is to write
\begin{align*}
\int_\Omega \psi_j(x,t) \cdot b_j(x,t) \, dx
&= \int_0^t \int_\Omega [\partial_\tau \psi_j(x,\tau) \cdot b_j(x,\tau) + \psi_j(x,\tau) \cdot \partial_\tau b_j(x,\tau)] \, dx \, d\tau \\
&+ \int_\Omega \psi_j(x,0) \cdot b_j(x,0) \, dx
\end{align*}
and use the induction equation
\begin{equation} \label{Resistive MHD2 copy}
\partial_t b_j + \curl (b_j \times u_j) + \mu_j \curl \curl b_j = 0
\end{equation}
 on $\partial_\tau \psi_j$ and $\partial_\tau b_j$. In the multiply connected case, however, \eqref{Resistive MHD2 copy} leads (formally) to $\partial_t \psi_j = - b_j \times u_j - \mu_j \curl b_j + \sum_{i=1}^N d_i(t) h_i + \nabla g$, where $\{h_1,\ldots,h_N\}$ is an orthonormal basis of $L^2_H(\Omega,\R^3)$, and the product $\sum_{i=1}^N d_i(t) h_i \cdot b_j$ seems very difficult to control. We therefore maneuver carefully in the proof of Lemma \ref{Distributional derivative of magnetic helicity} to make sure that \emph{we do not multiply $\partial_t \psi_j$ and $b_{j,H}$ at any point of the argument}.

\vspace{0.3cm}
The considerations above prompt us to decompose $\psi_j$ and take advantage of the differences between $b_{j,\Sigma}$ and $b_{j,H}$. Using the notation of Corollary \ref{Time-dependent Borchers-Sohr}, we write
\begin{equation} \label{Decomposition of psi}
\psi_j = (\psi_j - \Psi^\Sigma_j - \psi^H_j) + \Psi^\Sigma_j + \psi^H_j.
\end{equation}
In \eqref{Decomposition of psi}, $\partial_t \psi_j^H = 0$, while the condition $\Psi_j^\Sigma \times n|_\Gamma = 0$ ensures that many natural integrations by parts do not create unwanted boundary terms. These properties play a key role in the proof of Lemma \ref{Distributional derivative of magnetic helicity}. Finally, the 'bad part' $\psi_j-\psi_j^\Sigma-\psi_j^H$ is curl-free and, consequently, $\int_\Omega (\psi_j(x,t)-\Psi_j^\Sigma(x,t)-\psi_j^H(x)) \cdot b_j(x,t) \, dx$ can be given a simple formula (see Lemma \ref{Reduction lemma}).

Furthermore, while $\psi_j$ need not converge to $\psi$ in any useful sense, the 'good parts' $\Psi^\Sigma_j + \psi^H_j$ of the potentials satisfy
\begin{equation} \label{Strong convergence of potentials}
\Psi_j^\Sigma + \psi^H_j \to \Psi^\Sigma + \psi^H \qquad \text{in } L^2_{loc}(0,T;L^2(\Omega,\R^3))
\end{equation}
(see Lemma \ref{Lemma on strong limit of vector potentials}). The basic reason is that $\Psi^\Sigma_j$ and $\psi^H_j$ depend linearly on $b_j$, which allows us to exploit the weak-$*$ convergence $b_j \overset{*}{\rightharpoonup} b$ in $L^\infty(0,T;L^2_\sigma(\Omega,\R^3))$.

\subsection{An overview of the proof} \label{Strategy of the proof}
The proof of Theorem \ref{Taylor theorem} is reduced to the special case of $\psi_j = \Psi^\Sigma_j + \psi^H_j$ and $\psi = \Psi^\Sigma + \psi^H$ in Lemma \ref{Reduction lemma}. We therefore introduce a shorthand notation for magnetic helicity in this gauge.

\begin{defin} \label{Notation for canonical magnetic helicity}
Given $v \in L^\infty(0,T;L^2_\sigma(\Omega,\R^3))$ and $\Psi^\Sigma \defeq T_\Sigma v$, $\psi^H \defeq T_H v$ we denote
\[\mathscr{M}(v;t) \defeq \int_\Omega (\Psi^\Sigma(x,t) + \psi^H(x,t)) \cdot v(x,t) \, dx.\]
\end{defin}
Our aim is to show that
\begin{equation} \label{Formula of M for Leray-Hopf solutions}
\mathscr{M}(b_j;t) = \mathscr{M}(b_j;0) - 2 \mu_j \int_0^t \int_\Omega b_j(x,\tau) \cdot \curl b_j(x,\tau) \, dx \, d\tau
\end{equation}
for every $j \in \N$ and $t \in [0,T)$ and that given $\eta \in C_c^\infty(0,T)$,
\begin{equation} \label{Double aim}
\int_0^T \eta(t) \mathscr{M}(b;t) \, dt
= \lim_{j \to \infty} \int_0^T \eta(t) \mathscr{M}(b_j;t) \, dt
= \lim_{j \to \infty} \int_0^T \eta(t) \mathscr{M}(b_j;0) \, dt.
\end{equation}
Once \eqref{Double aim} is proved, \eqref{Value of magnetic helicity} follows for $\psi = \Psi^\Sigma + \psi^H$ rather easily (see Lemma \ref{Lemma on magnetic helicities of initial datas}).

The leftmost equality in \eqref{Double aim} is proved by showing \eqref{Strong convergence of potentials} and recalling that $b_j \rightharpoonup b$ in $L^2(0,T;L^2(\Omega,\R^3))$. The proof of \eqref{Strong convergence of potentials} uses the Aubin-Lions Lemma as a main tool and is presented in \textsection \ref{Strong convergence of the vector potentials}. The rightmost equality of \eqref{Double aim} is proved in \textsection \ref{Completion of the proof} by showing that the double integral on the right-hand side of \eqref{Formula of M for Leray-Hopf solutions} vanishes at the limit $j \to \infty$.

We finally mention that in the proof of Theorem \ref{Taylor theorem} we will on several occasions pass to a subsequence without relabeling it. The limit \eqref{Value of magnetic helicity} will however hold for the whole sequence $(b_j)_{j=1}^\infty$ as every subsequence will have a subsequence satisfying \eqref{Value of magnetic helicity}.

\subsection{Reduction to good vector potentials} \label{Reduction to good vector potentials}
The following lemma shows that it suffices to prove the claims of Theorem \ref{Taylor theorem} and Corollary \ref{Taylor corollary} for the potentials of Definition \ref{Notation for canonical magnetic helicity}. It also indicates to what extent gauge invariance of magnetic helicity fails in multiply connected domains.

\begin{lem} \label{Reduction lemma}
Suppose $\psi = \psi^\Sigma + \psi^H \in L^\infty(0,T;W^{1,2}(\Omega,\R^3))$ is a vector potential of $v \in L^\infty(0,T;L^2_\sigma(\Omega,\R^3))$. Then
\[\int_\Omega \psi(x,t) \cdot v(x,t) \, dx = \mathscr{M}(v;t) - \int_\Gamma \psi^\Sigma(x,t) \times n \cdot \psi^H(x) \, dx\]
at a.e. $t \in (0,T)$.
\end{lem}

\begin{proof}
By using the definition of $\mathscr{M}(v;t)$ and the facts that $\psi^\Sigma - \Psi^\Sigma \in \ker \curl = L^2_\Sigma(\Omega,\R^3)^\perp$  and $\Psi^\Sigma \times n|_\Gamma = 0$ we obtain
\begin{align*}
\int_\Omega \psi(x,t) \cdot v(x,t) \, dx
&= \mathscr{M}(v;t) + \int_\Omega (\psi^\Sigma(x,t) - \Psi^\Sigma(x,t)) \cdot v_H(x,t) \, dS(x) \\
&= \mathscr{M}(v;t) + \int_\Gamma \psi^H(x) \times n \cdot (\psi^\Sigma(x,t) -\Psi^\Sigma) \, dS(x) \\
&= \mathscr{M}(v;t) - \int_\Gamma \psi^\Sigma(x,t) \times n \cdot \psi^H(x) \, dx
\end{align*}
at a.e. $t \in (0,T)$.
\end{proof}

\subsection{Stationarity of the harmonic parts of magnetic fields}
We next show that for Leray-Hopf solutions and their weak ideal limits, the harmonic part of the magnetic field is stationary.

\begin{lem} \label{Decomposition of the magnetic field}
Under the assumptions of Theorem \ref{Taylor theorem}, for every $j \in \N$ the harmonic parts $b_{j,H}$ and $b_H$ are of the forms $b_{j,H}(x,t) = b_{j,0,H}(x)$ and $b_H(x,t) = b_{0,H}(x) = \lim_{j \to \infty} b_{j,0,H}(x)$.
\end{lem}

\begin{proof}
Given $j \in \N$ we write $b_{j,H}(x,t) = \sum_{i=1}^N c^j_i(t) h_i(x)$ and set out to prove that $c^1_j,\ldots,c^N_j$ are constants independent of $t$. We fix $i \in \{1,\ldots,N\}$ and first show that $c^i_j$ is continuous. Since $b_j$ is weakly $L^2$-continuous in time, we get
\[c_i^j(t_k) = \int_\Omega b_j(x,t_k) \cdot h_i(x) \, dx \to \int_\Omega b_j(x,t) \cdot h_i(x) \, dx = c_i^j(t)\]
whenever $t_k \to t$ in $[0,T)$. Now let $\eta \in C_c^\infty(0,T)$ and set $\theta(x,t) \defeq \eta(t) h_i(x)$ in \eqref{Resistive MHD weak definition 2}, getting $\int_0^T \eta'(t) c^j_i(t) \, dt = 0$, which implies that $c_i^j(t) = c_i^j(0)$ for all $t \in [0,T)$. Fixing $i \in \{1,\ldots,N\}$ and $\eta \in C_c^\infty(0,T)$ we get
\[c^j_i(0) \int_0^T \eta(t) \, dt
= \int_0^T \int_\Omega \eta(t) h_i(x) \cdot b_j(x,t) \, dx \, dt \to \int_0^T \int_\Omega \eta(t) h_i(x) \cdot b(x,t) \, dx \, dt,\]
which yields the statement on $b_H$.
\end{proof}

\subsection{Strong convergence of good vector potentials} \label{Strong convergence of the vector potentials}
The aim of this section is to prove \eqref{Strong convergence of potentials} via the Aubin-Lions Lemma. This requires uniform control of the norms $\norm{\partial_t \psi_j}_{L^1(0,T;X)}$ in some (reflexive) Banach space $X \supset L^2(\Omega,\R^3)$. Note that since $\partial_t \psi_j^H = 0$, \eqref{Resistive MHD2} yields
\begin{equation} \label{Vanishing curl}
\curl(\partial_t \psi^\Sigma_j + b_j \times u_j + \mu_j \curl b_j) = 0.
\end{equation}
If $\Omega' \subset \Omega$ is a simply connected subdomain, we can thus write $\partial_t \psi^\Sigma_j + b_j \times u_j + \mu_j \curl b_j = \nabla g$ in $\Omega'$. It is, however, not immediately clear how well-behaved $\partial_t \psi_j^\Sigma$ (and, thus, $g$) is. In order to circumvent this issue we mollify in time via the functions $t \mapsto \chi^\delta(t)$ mentioned in \textsection \ref{Bochner spaces} and write $\partial_t (\psi_j^\Sigma * \chi^\delta) = \psi_j^\Sigma * \partial_t \chi^\delta$. 

\begin{lem} \label{Lemma on strong limit of vector potentials}
The vector potentials $\Psi_j^\Sigma + \psi_j^H$ converge in $L^2_{loc}(0,T;L^2(\Omega,\R^3))$ to the vector potential $\Psi^\Sigma + \psi^H$ of $b$.
\end{lem}

\begin{proof}
Lemma \ref{Decomposition of the magnetic field} implies that $\psi_j^H$ converges to $\psi^H$ in $L^2(0,T;L^2(\Omega,\R^3))$. The more elaborate part is the strong convergence of the potentials $\Psi_j^\Sigma$.

We fix a sequence of numbers $\epsilon_j \in (0,T/2)$. We then choose another sequence of numbers $\delta_j \in (0,\epsilon_j)$ such that $\norm{\Psi^\Sigma_j * \chi^{\delta_j} - \Psi^\Sigma_j}_{L^2(\epsilon_j,T-\epsilon_j)} < 1/j$ for all $j \in \N$, so that it suffices to prove the convergence of the sequence $(\Psi_j^\Sigma * \chi^{\delta_j})_{j=1}^\infty$ in $L^2_{loc}(0,T;;L^2(\Omega,\R^3))$. We fix a non-empty simply connected, smooth subdomain $\Omega' \subset \Omega$ and aim to show that
\begin{equation} \label{Aim for Aubin-Lions}
\sup_{j \in \N} (\|\Psi_j^\Sigma * \chi^{\delta_j}\|_{L^2(\epsilon,T-\epsilon; W^{1,2}(\Omega,\R^3))} + \|\partial_t \Psi_j^\Sigma * \chi^{\delta_j}\|_{L^2(\epsilon,T-\epsilon; (W^{1,4}_{0,\sigma}(\Omega',\R^3))^*)}) < \infty;
\end{equation}
the Aubin-Lions Lemma then gives norm convergence of a subsequence of $(\Psi_j^\Sigma)_{j=1}^\infty$ to some $\tilde{\psi} \in L^2_{loc}(0,T;L^2(\Omega))$. On the other hand, for a further subsequence, $\Psi_j^\Sigma = T_\Sigma b_{j,\Sigma} \rightharpoonup T_\Sigma b_\Sigma$ in $L^2(0,T;W^{1,2}(\Omega,\R^3))$, and thus $\tilde{\psi} = T_\Sigma b_\Sigma \eqdef \Psi^\Sigma$.

For \eqref{Aim for Aubin-Lions} we fix $j \in \N$ and use the fact that by Lemma \ref{Borchers-Sohr lemma}, Theorem \ref{Theorem on Leray-Hopf solutions} and Lemma \ref{Lemma on time integrability of mollified functions},
\begin{align*}
\|\Psi_j^\Sigma* \chi^{\delta_j}\|_{L^2(\epsilon,T-\epsilon; W^{1,2}(\Omega,\R^3))}
&\lesssim \|b_{j,\Sigma} * \chi^{\delta_j}\|_{L^2(\epsilon,T-\epsilon;L^2(\Omega))} \le \|b_{j,\Sigma}\|_{L^2(0,T;L^2(\Omega))} \\
&\le \|b_j\|_{L^2(0,T;L^2(\Omega))}
\le \sqrt{T} \|b_{j,0}\|_{L^2(\Omega)}.
\end{align*}
In order to control the norm of $\partial_t (\Psi_j^\Sigma * \chi^{\delta_j})$ we note that $\partial_t (\Psi_j * \chi^{\delta_j}) = \Psi_j^\Sigma * \partial_t \chi^{\delta_j} \in L^\infty(\epsilon,T-\epsilon;W^{1,2}_0(\Omega'))$, and so Lemma \ref{Helmholtz-Hodge in domains} and \eqref{Vanishing curl} yield
\begin{equation} \label{Evolution of the vector potential with resistivity}
\Psi_j^\Sigma * \partial_t \chi^{\delta_j} + (b_j \times u_j) * \chi^{\delta_j} - \mu_j \curl b_j * \chi^{\delta_j} = \nabla g_j,
\end{equation}
where $g_j \in L^\infty(\epsilon,T-\epsilon;W^{1,2}(\Omega'))$. We estimate, at every $t \in (\epsilon,T-\epsilon)$,
\begin{align*}
   \left| \int_{\Omega'} \partial_t [\Psi_j^\Sigma * \chi^{\delta_j}](x,t) \cdot v(x) \, dx \right|
&\le \|b_j \times u_j * \chi^{\delta_j}(\cdot,t)\|_{L^1(\Omega)} \|v\|_{L^\infty(\Omega)} \\
&+ \mu_j \|\curl b_j * \chi^{\delta_j}(\cdot,t)\|_{L^2(\Omega)} \|v\|_{L^2(\Omega)}
\end{align*}
for all $v \in W^{1,4}_{0,\sigma}(\Omega',\R^3)$ so that, by Lemma \ref{Lemma on time integrability of mollified functions},
\begin{align*}
\|\partial_t \Psi_j^\Sigma * \chi^{\delta_j}\|_{L^2(\epsilon,T-\epsilon; (W^{1,4}_{0,\sigma}(\Omega))^*)}
&\lesssim_\Omega \|b_j \times u_j * \chi^{\delta_j}\|_{L^2(\epsilon,T-\epsilon; L^1(\Omega))} \\
&+ \mu_j \|\curl b_j * \chi^{\delta_j}\|_{L^2(\epsilon,T-\epsilon;L^2(\Omega))} \\
&\le \|b_j \times u_j\|_{L^2(0,T;L^1(\Omega))} + \mu_j \|\curl b_j\|_{L^2(0,T;L^2(\Omega))}
\end{align*}
which yields \eqref{Aim for Aubin-Lions}.
\end{proof}

\subsection{Completion of the proof} \label{Completion of the proof}
The proof of Theorem \ref{Taylor theorem} will be finished in the following two lemmas by showing \eqref{Formula of M for Leray-Hopf solutions} and controlling the size of the integral term in \eqref{Formula of M for Leray-Hopf solutions}. A third lemma then proves Corollary \ref{Taylor corollary}.

\begin{lem} \label{Distributional derivative of magnetic helicity}
For every $j \in \N$ and every $t \in [0,T)$ we have
\[\mathscr{M}(b_j;t) = \mathscr{M}(b_j;0) - 2 \mu_j \int_0^t \int_\Omega b_j(x,\tau) \cdot \curl b_j(x,\tau) \, dx \, d\tau.\]
\end{lem}

\begin{proof}
We intend to show that
\begin{equation} \label{Aimed distribution identity}
\partial_t \mathscr{M}(b_j;t) = - 2 \mu_j \int_\Omega b_j(x,t) \cdot \curl b_j(x,t) \, dx
\end{equation}
in the sense of distributions; the claim then follows since the Cauchy-Schwarz inequality gives $\partial_t \mathscr{M}(b_j;\cdot) \in L^1(0,T)$.

Let $\eta \in C_c^\infty(0,T)$ and note that $\int_0^T \eta(t) \mathscr{M}(b_j;t) \, dt = \lim_{\delta \to 0} \int_0^T \eta(t) \mathscr{M}(b_{j,\delta};t) \, dt$ by Lemma \ref{Mollifier approximation lemma}. Fix $\delta \in (0,T/2)$ such that $\supp(\eta) \subset [2\delta,T-2\delta]$. Then, integrating by parts several times and using the facts that $\Psi_j^\Sigma \times n|_\Gamma = 0$ and $\partial_t \Psi_j^H = \partial_t b_{j,H} = 0$ we get
\begin{align*}
\int_0^T \partial_t \eta(t) \mathscr{M}(b_{j,\delta};t) \, dt
&= \int_0^T \partial_t \eta(t) \int_\Omega (\Psi_j^\Sigma + \psi_j^H)_\delta(x,t) \cdot (b_{j,\Sigma} + b_{j,H})_\delta(x,t) \, dx \, dt \\
&= -2 \int_0^T \eta(t) \int_\Omega \partial_t \Psi_{j,\delta}^\Sigma(x,t) \cdot b_{j,\Sigma,\delta}(x,t) \, dx \, dt \\
&+ 2 \int_0^T \partial_t \eta_\delta(t) \int_\Omega \psi_j^H(x) \cdot b_{j,\Sigma}(x,t) \, dx \, dt \eqdef I_1 + I_2.
\end{align*}
For $I_1$ we note that $\partial_t \Psi_{j,\delta}^\Sigma + (b_j \times u_j)_\delta + \mu_j \curl b_{j,\delta} \in L^\infty(\delta,T-\delta;\ker(\curl)) = L^\infty(\delta,T-\delta;(L^2_\Sigma(\Omega,\R^3))^\perp)$ so that
\[I_1 = 2 \int_0^T \eta(t) \int_\Omega [(b_j \times u_j)_\delta(x,t) + \mu_j \curl b_{j,\delta}(x,t)] \cdot (b_{j,\delta}^\Sigma(x,t) \, dx \, dt.\]
For $I_2$ we note that since $\partial_t (\psi_j^H \cdot b_{j,H}) = 0$, we get
\[I_2
= 2 \int_0^T \partial_t \eta_\delta(t) \int_\Omega \psi_j^H(x) \cdot b_j(x,t) \, dx \, dt.\]
By setting $\theta(x,t) \defeq 2 \eta_\delta(t) \psi_j^H(x)$ in \eqref{Resistive MHD weak definition 2} we obtain
\begin{align*}
   I_2
&= 2 \int_0^T \eta_\delta(t) \int_\Omega \curl \psi_j^H(x) \cdot [b_j \times u_j(x,t) + \mu_j \curl b_j(x,t)] \, dx \, dt \\
&= 2 \int_0^T \eta(t) \int_\Omega b_j^H(x) \cdot [(b_j \times u_j)_\delta(x,t) + \mu_j \curl b_{j,\delta}(x,t)] \, dx \, dt.
\end{align*}
Collecting the identities, taking the limit $\delta \to 0$ (via Lemma \ref{Interpolation lemma}) and using the pointwise identity $b_j \cdot b_j \times u_j = 0$ we conclude that
\[\int_0^T \partial_t \eta(t) \mathscr{M}(b_j;t) \, dt = 2 \mu_j \int_0^T \eta(t) \int_\Omega b_j(x,t) \cdot \curl b_j(x,t) \, dx,\]
which yields \eqref{Aimed distribution identity}.
\end{proof}

The following estimate, which goes back to ~\cite{Berger}, completes the proof of Theorem \ref{Taylor theorem}.

\begin{lem}
For every $j \in \N$,
\[\mu_j \int_0^T \int_\Omega |b_j(x,t) \cdot \curl b_j(x,t)| \, dx \, dt \lesssim_T \sqrt{\mu_j} (\|u_{j,0}\|_{L^2}^2 + \|b_{j,0}\|_{L^2}^2).\] 
\end{lem}

\begin{proof}
By Young's inequality and the energy inequality,
\[\begin{array}{lcl}
& & \displaystyle \mu_j \int_0^T \int_\Omega |b_j(x,t) \cdot \curl b_j(x,t)| \, dx \, dt \\
&\le& \displaystyle \frac{\sqrt{\mu_j}}{2} \int_0^T \int_\Omega (|b_j(x,t)|^2 + \mu_j |\curl b_j(x,t)|^2) \, dx \, dt \\
&\lesssim_T& \displaystyle \sqrt{\mu_j} (\|u_{j,0}\|_{L^2}^2 + \|b_{j,0}\|_{L^2}^2).
\end{array}\]
\end{proof}

A simple lemma gives the rightmost equality in \eqref{Magnetic helicity for good potentials}.

\begin{lem} \label{Lemma on magnetic helicities of initial datas}
$\lim_{j \to \infty} \int_\Omega (\Psi^\Sigma_{j,0}(x) + \psi^H_{j,0}(x)) \cdot b_{j,0}(x) \, dx = \int_\Omega (\Psi^\Sigma_{j,0}(x) + \psi^H_{j,0}(x)) \cdot b_0(x) \, dx$.
\end{lem}

\begin{proof}
By assumption, $b_{j,0} \rightharpoonup b_0$ in $L^2_\sigma(\Omega,\R^3)$, and therefore $\Psi^\Sigma_{j,0} + \psi^H_{j,0} = T_\Sigma b_{j,0,\Sigma} + T_H b_{j,0,H} \rightharpoonup T_\Sigma b_{0,\Sigma} + T_H b_{0,H} = \Psi^\Sigma_0 + \psi^H_0$ in $W^{1,2}(\Omega,\R^3)$. The Rellich-Kondrachov Theorem then yields $\|(\Psi^\Sigma_{j,0} + \psi^H_{j,0})-(\Psi^\Sigma_0 + \psi^H_0)\|_{L^2(\Omega)} \to 0$, which implies the claim.
\end{proof}

\section{A two-dimensional analogue} \label{A two-dimensional analogue}
Magnetic helicity has a two-dimensional counterpart, the mean-square magnetic potential. It is defined as the $L^2$ energy of the canonical stream function of $b$, and it is conserved in time by smooth solutions of ideal 2D MHD. In \textsection \ref{Mean-square magnetic potential and statement of the theorem} we define the mean-square magnetic potential in multiply connected domains and formulate Theorem \ref{2D Taylor theorem} which says that it is also conserved in the weak ideal limit. As a byproduct, we prove that if a weak solution of 2D ideal MHD lies in the energy space, then it conserves mean-square magnetic potential in time. The proof of Theorem \ref{2D Taylor theorem} is presented in \textsection \ref{Proof of 2D TAylor theorem}. As main tools, apart from ones already used in 3D, we use C. Fefferman's $\mathcal{H}^1$--$\operatorname{BMO}$ duality theorem from ~\cite{FS72} and the Hardy space theory of compensated compactness quantities of Coifman, Lions, Meyer and Semmes from ~\cite{CLMS93}.

\subsection{Mean-square magnetic potential and statement of the theorem} \label{Mean-square magnetic potential and statement of the theorem}
In two dimensions, the \emph{viscous, resistive MHD equations} are given by
\begin{align}
& \partial_t u + (u \cdot \nabla) u - (b \cdot \nabla) b - \nu \Delta u+ \nabla \Pi = 0, \label{2D resistive MHD} \\
& \partial_t b - \nabla^\perp (b \times u) - \mu \nabla^\perp (\curl b) = 0, \label{2D resistive MHD2} \\
& \dive u = \dive b = 0, \label{2D resistive MHD3} \\
& u(\cdot,0) = u_0, \; b(\cdot,0) = b_0, \label{2D resistive MHD4}
  \end{align}
where $\nabla^\perp = (-\partial_2, \partial_1)$ and $\curl = \nabla^\perp \cdot$.
We now record our assumptions on the domain; we weaken the regularity condition that we placed on the boundary in three dimensions. Assumption \ref{Assumption 3 on Omega} is strong enough to ensure the existence of a canonical stream function for every vector field in $L^2_\sigma(\Omega,\R^2)$ (see Theorem \ref{Stream function theorem}).

\begin{assumption} \label{Assumption 3 on Omega}
The set $\Omega \subset \R^2$ is open and bounded. Furthermore, $\Omega$ is connected and its boundary $\Gamma$ is Lipschitz-continuous and has a finite number of connected components denoted by $\Gamma_1,\ldots,\Gamma_K$.
\end{assumption}

The boundary conditions corresponding to \eqref{Resistive MHD5}--\eqref{Resistive MHD6} are
\begin{align}
& u|_\Gamma = 0, \label{2D resistive MHD5} \\
& b \cdot n|_\Gamma = 0 \qquad \text{and} \qquad \curl b|_\Gamma = 0. \label{2D resistive MHD6}
\end{align}
Equations \eqref{2D resistive MHD}--\eqref{2D resistive MHD6} are understood in analogy to the 3D case, but \eqref{Resistive MHD weak definition 2a} needs to be replaced by the formula
\begin{equation} \label{2D momentum equation interpretation} 
\langle \partial_t b, \theta \rangle + \int_\Omega b \times u \curl \theta + \mu \int_\Omega \curl b \curl \theta = 0.
\end{equation}

We enumerate $\Gamma_1,\ldots,\Gamma_K$ in such a way that $\Gamma_1$ is the boundary of the unbounded component of $\R^2 \setminus \overline{\Omega}$. Following ~\cite[p. 40]{GR} we denote
\[\Phi \defeq \{\psi \in W^{1,2}(\Omega) \colon \psi|_{\Gamma_1} = 0, \; \psi|_{\Gamma_i} \text{ is constant for } 2 \le i \le K\};\]
note that if $\Gamma$ is connected, then $\Phi = W^{1,2}_0(\Omega)$.
The following theorem gives a canonical choice of stream functions (see ~\cite[Corollary I.3.1]{GR}).

\begin{thm} \label{Stream function theorem}
The mapping $-\nabla^\perp \colon \Phi \to L^2_\sigma(\Omega,\R^2)$ is an isomorphism.
\end{thm}

When $v \in L^2_\sigma(\Omega,\R^2)$, we call $(-\nabla^\perp)^{-1} v \in \Phi$ the \emph{stream function of} $v$. Leray-Hopf solutions are defined in direct analogy to Definition \ref{Leray-Hopf solutions}.

\begin{defin}
Suppose $(u,b)$ is a Leray-Hopf solution of \eqref{2D resistive MHD}--\eqref{2D resistive MHD6} and $\psi \in C_w([0,T);\Phi)$ is the stream function of $b$. For every $t \in [0,T)$, $\int_\Omega \abs{\psi(x,t)}^2 dx$ is called the \emph{mean-square magnetic potential of $b$ at time} $t$.
\end{defin}

We formulate an analogue of Theorem \ref{Taylor theorem} for the mean-square magnetic potential, denoting the stream functions of the initial datas $b_{j,0}$ and $b_0$ by $\psi_{j,0}$ and $\psi_0$. The weak ideal limit and weak non-resistive limit are defined in direct analogy to Definition \ref{Definition of weak ideal limit}.

\begin{thm} \label{2D Taylor theorem}
Suppose $\Omega \subset \R^2$ satisfies Assumption \ref{Assumption 3 on Omega}, and assume that $u,b \in L^\infty(0,T;L^2_\sigma(\Omega,\R^3))$ are a weak ideal limit or weak non-resistive limit of Leray-Hopf solutions $(u_j,b_j)$, $j \in \N$. Then $b \in C_w([0,T);L^2_\sigma(\Omega,\R^2))$, $\partial_t b - \nabla^\perp (b \times u) = 0$ with $b(\cdot,0) = b_0$ and
\begin{equation} \label{Limit of magnetic mean square potentials}
\int_\Omega \abs{\psi(x,t)}^2 dx = \int_\Omega \abs{\psi_0(x)}^2 dx = \lim_{j \to \infty} \int_\Omega \abs{\psi_{j,0}(x)}^2 dx
\end{equation}
for all $t \in [0,T)$.
\end{thm}

Note that Theorem \ref{2D Taylor theorem} is stronger than Theorem \ref{Taylor theorem} in the sense that the induction equation $\partial_t b - \nabla^\perp (b \times u) = 0$ holds in the weak ideal limit. Theorem \ref{2D Taylor theorem} is proved in the following subsection.

\subsection{Proof of Theorem \ref{2D Taylor theorem}} \label{Proof of 2D TAylor theorem}
Our first task is to prove that the induction equation $\partial_t b - \nabla^\perp (b \times u) = 0$ holds and $b(\cdot,0) = b_0$. We begin the proof by showing a 2D analogue of Lemma \ref{Lemma on strong limit of vector potentials}.

\begin{lem} \label{Lemma on strong limit of stream functions}
$\psi_j \to \psi$ in $L^2(0,T;L^2(\Omega))$.
\end{lem}

\begin{proof}
Since $-\nabla^\perp \psi_j = b_j \rightharpoonup b = - \nabla^\perp \psi$ in $L^2(0,T;L^2_\sigma(\Omega,\R^2))$, we have $\psi^j \rightharpoonup \psi$ in $L^2(0,T;\Phi)$ by Theorem \ref{Stream function theorem}. Hence, it suffices, by the Aubin-Lions Lemma, to show that
\begin{equation} \label{Aim 2 for Aubin-Lions}
\sup_{j \in \N} \|\partial_t \psi_j\|_{L^1(0,T;W^{-1,2}(\Omega))} < \infty.
\end{equation}
We write $\partial_t \psi_j = - b_j \times u_j - \mu_j \curl b_j$ and estimate the terms separately. First, we set
\begin{equation} \label{Zero extensions}
B_j(x,t) \defeq \begin{cases}
                    b_j(x,t), & x \in \Omega, \\
                    0, & x \notin \Omega,
                  \end{cases} \quad
  U_j(x,t) \defeq \begin{cases}
                    u_j(x,t), & x \in \Omega, \\
                    0, & x \notin \Omega
                  \end{cases}
\end{equation}
and note that $\dive B_j = \dive U_j = 0$ in $\R^2 \times (0,T)$. Fix $t \in (0,T)$ and $\varphi \in C_c^\infty(\Omega)$, and denote $\Phi(x) = \varphi(x)$ for $x \in \Omega$ and $\Phi(x) = 0$ for $x \notin \Omega$. Fefferman's $\mathcal{H}^1$-$\operatorname{BMO}$ duality theorem and the $\dive$-$\curl$ estimate of Coifman, Lions, Meyer and Semmes give
\begin{align*}
\int_\Omega b_j(x,t) \times u_j(x,t) \varphi(x) \, dx
&= \int_{\R^2} B_j(x,t) \times U_j(x,t) \Phi(x) \, dx \\
&\lesssim \|B_j(\cdot,t) \times U_j(\cdot,t)\|_{\mathcal{H}^1} \|\Phi\|_{\operatorname{BMO}} \\
&\lesssim \|B_j(\cdot,t)\|_{L^2} \|U_j(\cdot,t)\|_{L^2} \|\nabla \Phi\|_{L^2},
\end{align*}
yielding, by the Cauchy-Schwartz inequality,
\[\sup_{j \in \N} \|b_j \times u_j\|_{L^1(0,T;W^{-1,2}(\Omega))} \lesssim \sup_{j \in \N} \norm{b_j}_{L^2(0,T;L^2(\Omega))} \norm{u_j}_{L^2(0,T;L^2(\Omega))} < \infty.\]
Furthermore, trivially, $\sup_{j \in \N} \|\mu_j \curl b_j\|_{L^1(0,T;W^{-1,2}(\Omega))} < \infty$, and so \eqref{Aim 2 for Aubin-Lions} holds.
\end{proof}

We next show that the limit mappings $u$ and $b$ satisfy the ideal momentum equation. Given $j \in \N$, the mappings $u_j$ and $b_j$ satisfy \eqref{2D resistive MHD}--\eqref{2D resistive MHD6} and using standard arguments (see e.g. \cite[Lemma 2.4]{Gal}), \eqref{2D momentum equation interpretation} and the initial value condition $b_j(\cdot,0) = b_{j,0}$ yield
\begin{equation} \label{2D alternative formula for weak solutions}
\int_0^T \int_\Omega (b_j \cdot \partial_t \phi - b_j \times u_j \curl \phi - \mu \curl b_j \curl \phi) \, dx \, dt + \int_\Omega b_{j,0} \cdot \phi(\cdot,0)
\end{equation}
for every $\phi \in C_c^\infty(\Omega \times [0,T),\R^2)$ with $\dive \phi = 0$.

\begin{lem}
$\partial_t b - \nabla^\perp (b \times u) = 0$ with initial value $b(\cdot,0) = b_0$.
\end{lem}

\begin{proof}
Fix $\phi \in C_c^\infty(\Omega \times [0,T),\R^2)$ with $\dive \phi = 0$. By using the formula $b_j \times u_j = \nabla \psi_j \cdot u_j$, Theorem \ref{Helmholtz-Weyl decomposition} and Lemma \ref{Lemma on strong limit of stream functions},
\[\begin{array}{lcl}
& & \displaystyle \int_0^T \int_\Omega b_j(x,t) \times u_j(x,t) \curl \phi(x,t) \, dx \, dt \\
&=& \displaystyle \int_0^T \int_\Omega \nabla [\curl \phi(x,t) \psi_j(x,t)] \cdot u_j(x,t) \, dx \, dt \\
&-& \displaystyle \int_0^T \int_\Omega \psi_j(x,t) \nabla \curl \phi(x,t) \cdot u_j(x,t) \, dx \, dt \\
&\to& \displaystyle - \int_0^T \int_\Omega \psi(x,t) \nabla \curl \phi(x,t) \cdot u(x,t) \, dx \, dt \\
&=& \displaystyle \int_0^T \int_\Omega b(x,t) \times u(x,t) \curl \phi(x,t) \, dx \, dt.
\end{array}\]
The claim now follows immediately by inspection of \eqref{2D alternative formula for weak solutions}, since the energy inequality yields $\mu_j \int_0^T \int_\Omega \curl b_j(x,t) \curl \phi(x,t) \, dx \, dt \to 0$.
\end{proof}

As in 3D, by adapting ~\cite[Lemmas 2.1 and 2.2]{Gal} we may assume that $b \in C_w([0,T);L^2_\sigma(\Omega,\R^2))$. Thus the stream function $\psi$ belongs to $C_w([0,T); \Phi) \subset C([0,T); L^2(\Omega))$. 

Theorem \ref{2D Taylor theorem} will be proved once we show \eqref{Limit of magnetic mean square potentials}. The right equality in \eqref{Limit of magnetic mean square potentials} follows from the assumption $b_{j,0} \rightharpoonup b_0$ and the Rellich-Kondrachov Theorem. We next prove the left equality -- in fact, we also prove that every weak solution of ideal MHD in the energy space conserves magnetic helicity in time.

\begin{lem} \label{Lemma for ideal MHD}
Suppose $u \in L^\infty(0,T;L^2_\sigma(\Omega,\R^2))$ and $b \in C_w([0,T);L^2_\sigma(\Omega,\R^2))$ satisfy $\partial_t b - \nabla^\perp (b \times u) = 0$ with initial value $b(\cdot,0) = b_0 \in L^2_\sigma(\Omega,\R^2)$. Then $b$ conserves mean square magnetic potential in time.
\end{lem}

\begin{proof}
Since $\psi \in C([0,T);L^2(\Omega))$, it suffices to show that for every $\eta \in C_c^\infty(0,T)$ we have $\int_0^T \partial_t \eta(t) \int_\Omega \abs{\psi(x,t)}^2 dx \, dt = 0$. Fix such an $\eta$ and choose $\epsilon > 0$ such that $\supp(\eta) \subset [\epsilon,T-\epsilon]$. Whenever $0 < \delta < \epsilon$, we mollify in time and write $\psi_\delta \defeq \psi * \chi_\delta$. By Lemma \ref{Mollifier approximation lemma}, $\int_0^T \partial_t \eta(t) \int_\Omega \abs{\psi(x,t)}^2 dx \, dt = \lim_{\delta \to 0} \int_0^T \partial_t \eta(t) \int_\Omega \abs{\psi_\delta(x,t)}^2 dx \, dt$.

Let now $0 < \delta < \epsilon$. The induction equation $\partial_t b - \nabla^\perp (b \times u) = 0$ and the assumptions about the boundary values of $u$, $b$ and $\psi$ imply that $\partial_t \psi + b \times u = 0$, and thus $(b \times u)_\delta = - \psi * \partial_t \chi_\delta \in L^\infty(\epsilon,T-\epsilon; W^{1,2}(\Omega))$, giving
\[\int_0^T \partial_t \eta(t) \int_\Omega \abs{\psi_\delta(x,t)}^2 dx \, dt
= 2 \int_0^T \eta(t) \int_\Omega \psi_\delta(x,t) (b \times u)_\delta(x,t) \, dx \, dt.\]
As in \eqref{Zero extensions}, we denote the zero extensions of $b$ and $u$ outside $\Omega$ by $B$ and $U$. Likewise, for every $t \in [0,T)$, we denote by $\Psi(\cdot,t) \in W^{1,2}(\R^2)$ the unique compactly supported solution of $-\nabla^\perp \Psi(\cdot,t) = B(\cdot,t)$. Thus $\Psi \in L^\infty(0,T; W^{1,2}(\R^2)) \subset L^2(0,T;\operatorname{VMO}(\R^2))$ and $B \times U \in L^\infty(0,T;\mathcal{H}^1(\R^2)) \subset (L^2(0,T;\operatorname{VMO}(\R^2)))^*$. This allows us to write, using Lemma \ref{Mollifier approximation lemma},
\begin{align*}
   \int_0^T \eta(t) \int_\Omega \psi_\delta(x,t) (b \times u)_\delta(x,t) \, dx \, dt
&= \int_0^T \eta(t) \int_{\R^2} \Psi_\delta(x,t) (B \times U)_\delta(x,t) \, dx \, dt \\
&\to \int_0^T \eta(t) \langle \Psi(\cdot,t), B \times U(\cdot,t) \rangle_{\operatorname{VMO}-\mathcal{H}^1} dt.
\end{align*}
We finally mollify $\Psi$ and $B$ in \emph{space} and use the Dominated Convergence Theorem in time to conclude that
\[\begin{array}{lcl}
& & \displaystyle \int_0^T \eta(t) \langle \Psi(\cdot,t), B \times U(\cdot,t) \rangle_{\operatorname{VMO}-\mathcal{H}^1} dt \\
&=& \displaystyle \lim_{\epsilon \to 0} \int_0^T \eta(t) \int_{\R^2} \Psi_\epsilon(x,t) B_\epsilon(x,t) \times U(x,t) \, dx \, dt \\
&=& \displaystyle \frac{1}{2} \lim_{\epsilon \to 0} \int_0^T \eta(t) \int_{\R^2} \nabla \abs{\Psi_\epsilon(x,t)}^2 \cdot U(x,t) \, dx \, dt = 0.
\end{array}\]

\end{proof}

\appendix

\section{The existence of Leray-Hopf solutions in multiply connected domains}
We give a proof of the existence of Leray-Hopf solutions of \eqref{Resistive MHD}--\eqref{Resistive MHD6}, referring to the literature on some of the steps that are familiar from Navier-Stokes equations. A proof for simply connected domains is sketched in ~\cite{ST} and presented in more detail in ~\cite{GLBL}. As we cover multiply connected domains, more technicalities are needed although we follow the general scheme of the proof given in ~\cite{GLBL}. We reformulate Theorem \ref{Theorem on Leray-Hopf solutions} for the convenience of the reader.

\begin{thm} \label{Appendix theorem}
Suppose $\Omega$ satisfies Assumption \ref{Assumption on Omega} and let $u_0,b_0 \in L^2_\sigma(\Omega,\R^3)$. Then there exists a Leray-Hopf solution $(u,b)$ of \eqref{Resistive MHD}--\eqref{Resistive MHD6}.
\end{thm}

The basic strategy of the proof, via finite-dimensional \emph{Galerkin approximations}, is classical, but we discuss the main ideas. The solution is built via orthonormal bases $\{v_j\}_{j \in \N}$ and $\{w_j\}_{j \in \N}$ of $L^2_\sigma(\Omega,\R^3)$ satisfying the sought boundary conditions, that is, $v_j \in W^{1,2}_{0,\sigma}(\Omega,\R^3)$ and $w_j \in W^{1,2}_\sigma(\Omega,\R^3)$ with $(\curl w_j \times n)|_\Gamma = 0$.

\begin{defin} \label{Definition of Galerkin approximations}
Suppose $u_0, b_0 \in L^2_\sigma(\Omega,\R^3)$ and let $n \in \N$. 
Mappings of the forms
\begin{equation} \label{u_n and b_n}
u_n(x,t) = \sum_{j=1}^n c_{nj}(t) v_j(x), \quad
b_n(x,t) = \sum_{j=1}^n d_{nj}(t) w_j(x),
\end{equation}
where $c_{nj}, d_{nj} \in C^1([0,T))$, satisfy the $n$th order \emph{Galerkin approximation} of \eqref{Resistive MHD}--\eqref{Resistive MHD6} if
\begin{align*}
& \frac{d}{dt} (u_n,v_j)_{L^2} + \nu (\nabla u_n,\nabla v_j)_{L^2} + \langle (u_n \cdot \nabla) u_n - (b_n \cdot \nabla) b_n, v_j \rangle_{(W^{1,2}_{0,\sigma})^* - W^{1,2}_{0,\sigma}} = 0, \quad \\
& \frac{d}{dt} (b_n,w_j)_{L^2} + \mu (\curl b_n,\curl w_j)_{L^2} + \langle \curl (b_n \times u_n), w_j \rangle_{(W^{1,2}_\sigma)^* - W^{1,2}_\sigma} = 0, \\
& u_n(\cdot,0) = P_n u_0, \quad b_n(\cdot,0) = Q_n b_0
\end{align*}
for all $j = 1,\ldots,n$.
\end{defin}

For every $n \in \N$, standard theory of ordinary differential equations gives a unique solution of the Galerkin approximation satisfying the energy equality
\begin{equation} \label{Energy equality}
\begin{array}{lcl}
& & \displaystyle \frac{1}{2} \int_{\Omega} (\abs{u_n(x,t)}^2 + \abs{b_n(x,t)}^2) \, dx \\
&+& \displaystyle \int_0^t \int_{\Omega} (\nu \abs{\nabla u_n(x,\tau)}^2 + \mu \abs{\curl b_n(x,\tau)}^2) \, dx \, d\tau \\
&=& \displaystyle \frac{1}{2} \int_{\Omega} (\abs{P_n u_0(x)}^2 + \abs{Q_n b_0(x)}^2) \, dx
\end{array}
\end{equation}
for all $t \in (0,T)$ (see Lemma \ref{Lemma on Galerkin approximations}). With some work, the energy equality allows us to subtract a subsequence with $u_n \rightharpoonup u$ in $L^2(0,T;W^{1,2}_{0,\sigma}(\Omega,\R^3))$ and $b_n \rightharpoonup b$ in $L^2(0,T;W^{1,2}_\sigma(\Omega,\R^3))$. Our goal is to show that $(u,b)$ is a Leray-Hopf solution with initial data $(u_0,b_0)$.

\vspace{0.3cm}
For every $n \in \N$ we denote by $P_n$ and $Q_n$ the projections of $L^2_\sigma(\Omega,\R^3)$ onto $\operatorname{span} \{v_1,\ldots,v_n\}$ and $\operatorname{span} \{w_1,\ldots,w_n\}$. Note that $P_n \colon W^{1,2}_{0,\sigma}(\Omega,\R^3) \to W^{1,2}_{0,\sigma}(\Omega,\R^3)$ and $Q_n \colon W^{1,2}_\sigma(\Omega,\R^3) \to W^{1,2}_\sigma(\Omega,\R^3)$ are also bounded operators, and we denote their (Banach space) adjoints by $P_n^* \colon (W^{1,2}_{0,\sigma}(\Omega,\R^3))^* \to(W^{1,2}_{0,\sigma}(\Omega,\R^3))^*$ and $Q_n^* \colon (W^{1,2}_\sigma(\Omega,\R^3))^* \to (W^{1,2}_\sigma(\Omega,\R^3))^*$. We also define the Stokes operator and a corresponding operator for magnetic fields,
\[\Lambda_1 \colon W^{1,2}_{0,\sigma}(\Omega,\R^3) \to (W^{1,2}_{0,\sigma}(\Omega,\R^3))^* \quad \text{and} \quad \Lambda_2 \colon W^{1,2}_\sigma(\Omega,\R^3) \to (W^{1,2}_\sigma(\Omega,\R^3))^*,\]
by
\begin{align*}
& \langle \Lambda_1 u, v \rangle_{(W^{1,2}_{0,\sigma})^*-W^{1,2}_{0,\sigma}} \defeq \int_\Omega \nabla u \colon \nabla v, \\
& \langle \Lambda_2 b, w \rangle_{(W^{1,2}_\sigma)^*-W^{1,2}_\sigma} \defeq \int_\Omega \curl b \cdot \curl w.
\end{align*}
We write the Galerkin approximation in the condensed form
\begin{align}
& \partial_t u_n - \nu P_n^* \Lambda_1 u_n + P_n^* [(u_n \cdot \nabla) u_n - (b_n \cdot \nabla) b_n] = 0, \label{Galerkin abridged 1}\\
& \partial_t b_n - \mu Q_n^* \Lambda_2 b_n + Q_n^* [\curl(b_n \times u_n)] = 0, \label{Galerkin abridged 2} \\
& u_n(\cdot,0) = P_n u_0, \quad b_n(\cdot,0) = Q_n b_0. \label{Galerkin abridged 3}
\end{align}

In order for the weak limit $(u,b)$ to satisfy the MHD equations \eqref{Resistive MHD}--\eqref{Resistive MHD6} we need to gain enough compactness in the nonlinear terms $P_n^*[(u_n \cdot \nabla) \cdot u_n - b_n \times (\curl b_n)]$ and $Q_n^*[\curl (b_n \times u_n)]$. This is eventually achieved by using the Aubin Lions Lemma to get $u_n \to u$ and $b_n \to b$ in $L^2(0,T;L^2_\sigma(\Omega,\R^3))$. In order to satisfy the assumptions of the Aubin-Lions Lemma we wish to choose suitable bases $\{v_j\}_{j \in \N}$ and $\{w_j\}_{j \in \N}$ (see \textsection \ref{The choice of bases}) that ensure the uniform norm control
\begin{equation} \label{Uniform bound on operator norms}
\sup_{n \in \N} (\|P_n^*\|_{(W^{1,2}_{0,\sigma}(\Omega,\R^3))^* \to (W^{1,2}_{0,\sigma}(\Omega,\R^3))^*} + \|Q_n^*\|_{(W^{1,2}_\sigma(\Omega,\R^3))^* \to (W^{1,2}_\sigma(\Omega,\R^3))^*}) < \infty.
\end{equation}
As is customary, we select $v_j$ to be eigenfunctions of $\Lambda_1$, while a basis of $L^2_\Sigma(\Omega,\R^3)$ consists of eigenfunctions of $\Lambda_2$. Since we deal with multiply connected domains, $\{w_j\}_{j \in \N}$ also needs to include a basis of $L^2_H(\Omega,\R^3)$, and some care is needed in the ensuing arguments. The proof is completed in \textsection \ref{Passing to the limit}.

\subsection{The choice of bases} \label{The choice of bases}

This subsection is devoted to the choice of the orthonormal bases of $L^2_\sigma(\Omega,\R^3)$. The first basis $\{v_j\}_{j \in \N}$ consists of eigenfunctions of the Stokes operator and its existence is classical; we refer to ~\cite[p. 39]{Tem}. We endow $W^{1,2}_{0,\sigma}(\Omega,\R^3)$ with the Hilbert norm $\norm{\cdot}_{W^{1,2}_{0,\sigma}} \defeq \norm{\nabla \cdot}_{L^2}$.

\begin{lem} \label{Stokes lemma}
$L^2_\sigma(\Omega,\R^3)$ has an orthonormal basis $\{v_j\}_{j \in \N}$ with the following properties: for every $j \in \N$ there exists $\lambda_j > 0$ such that $v_j \in W^{1,2}_{0,\sigma}(\Omega,\R^3)$ satisfies
\[(v_j, \phi)_{W^{1,2}_{0,\sigma}} = \lambda_j (v_j, \phi)_{L^2}\]
for all $\phi \in W^{1,2}_{0,\sigma}(\Omega,\R^3)$. In particular, $\{v_j/\sqrt{\lambda_j}\}_{j \in \N}$ is an orthonormal system in $W^{1,2}_{0,\sigma}(\Omega,\R^3)$.
\end{lem}

The analysis of the second basis is simplified by using the following lemma which is essentially a special case of ~\cite[Corollary 3.16]{ABDG}.

\begin{lem} \label{Equivalent inner product}
On $W^{1,2}_\sigma(\Omega,\R^3)$, the norm $\norm{\cdot}_{W^{1,2}_\sigma}$ induced by the inner product
\[(v, w)_{W^{1,2}_\sigma} \defeq \int_\Omega \curl v(x) \cdot \curl w(x) \, dx + \sum_{i=1}^N \gamma_i \langle v \cdot n, 1 \rangle_{\Sigma_i} \langle w \cdot n, 1 \rangle_{\Sigma_i}\]
(where $\gamma_i > 0$ is chosen such that $\|h_i\|_{W^{1,2}_\sigma} = 1$ for all $i \in \{1,\ldots,N\}$) is equivalent to the norm inherited from $W^{1,2}(\Omega,\R^3)$.
\end{lem}

Lemma \ref{Equivalent inner product} has the following consequence (~\cite[Lemme II.6]{Dominguez}).

\begin{lem} \label{Orthogonal decomposition lemma}
The vector spaces $W^{1,2}_\Sigma(\Omega,\R^3) \defeq W^{1,2}_\sigma(\Omega,\R^3) \cap L^2_\Sigma(\Omega,\R^3)$ and $W^{1,2}_H(\Omega,\R^3) \defeq W^{1,2}_\sigma(\Omega,\R^3) \cap L^2_H(\Omega,\R^3)$ satisfy
\begin{equation} \label{Orthogonal decomposition of Sobolev space}
W^{1,2}_\sigma(\Omega,\R^3) = W^{1,2}_\Sigma(\Omega,\R^3) \oplus W^{1,2}_H(\Omega,\R^3).
\end{equation}
Furthermore, $W^{1,2}_\Sigma(\Omega,\R^3)$ is dense in $L^2_\Sigma(\Omega,\R^3)$.
\end{lem}

\begin{proof}
In the proof of \eqref{Orthogonal decomposition of Sobolev space} the only non-trivial condition to check is that when $w \in W^{1,2}_\sigma(\Omega,\R^3)$, we have $w_\Sigma \defeq P_\Sigma w \in W^{1,2}_\Sigma(\Omega,\R^3)$ and $w_H \defeq P_H w \in W^{1,2}_H(\Omega,\R^3)$. Note that Theorem \ref{Foias-Temam theorem} gives $w_H \in W^{1,2}_H(\Omega,\R^3)$, which immediately implies $w_\Sigma \in W^{1,2}_\Sigma(\Omega,\R^3)$. Furthermore, the projection $P_\Sigma \colon L^2_\sigma(\Omega,\R^3) \to L^2_\Sigma(\Omega,\R^3)$ is also a bounded operator from $W^{1,2}_\sigma(\Omega,\R^3)$ into $W^{1,2}_\Sigma(\Omega,\R^3)$.

Let now $f \in L^2_\Sigma(\Omega,\R^3)$ and choose mappings $\psi_j \in C_{c,\sigma}^\infty(\Omega,\R^3) \subset W^{1,2}_\sigma(\Omega,\R^3)$ such that $\|\psi_j - f\|_{L^2} \to 0$. Then $P_\Sigma \psi_j \in W^{1,2}_\Sigma(\Omega,\R^3)$ for all $j \in \N$ and $P_\Sigma \psi_j \to P_\Sigma f = f$ in $L^2(\Omega,\R^3)$.
\end{proof}

We use Lemma \ref{Orthogonal decomposition lemma} to find the basis of $L^2_\sigma(\Omega,\R^3)$ that is used to construct the magnetic field in Theorem \ref{Appendix theorem}. In the case of simply connected domains this is done by analysing the magnetostatic problem instead of the stationary Stokes problem (see ~\cite[pp. 67--69]{GLBL}). In multiply connected domains the situation is a bit more complicated because $W^{1,2}_H(\Omega,\R^3)$ is non-trivial.

\begin{lem} \label{Magnetostatic lemma}
$L^2_\sigma(\Omega,\R^3)$ has an orthonormal basis $\{w_j\}_{j \in \N}$ with the following properties: $\{w_1,\ldots,w_N\} = \{h_1,\ldots,h_N\}$ and for every $j \in \N$ there exists $\tilde{\lambda_j} > 0$ such that $w_j \in W^{1,2}_\sigma(\Omega,\R^3)$ satisfies
\begin{equation} \label{Magnetostatic equation}
(w_j, \psi)_{W^{1,2}_\sigma} = \tilde{\lambda}_j (w_j, \psi)_{L^2}
\end{equation}
for all $\psi \in W^{1,2}_\sigma(\Omega,\R^3)$. In particular, $\{w_j/\sqrt{\tilde{\lambda}_j}\}_{j \in \N}$ is an orthonormal system in $W^{1,2}_\sigma(\Omega,\R^3)$.
\end{lem}

\begin{proof}
Given $j \in \{1,\ldots,N\}$ we first check that \eqref{Magnetostatic equation} holds for $w_j = h_j$ with $\tilde{\lambda}_j = 1$. Let $\psi \in W^{1,2}_\sigma(\Omega,\R^3)$. Writing $\psi = \sum_{i=1}^N (\psi, h_i)_{L^2} h_i + \psi_\Sigma$ and using Lemma \ref{Orthogonal decomposition lemma}, \eqref{Magnetostatic equation} follows immediately.

Next we set out to find $w_j \in W^{1,2}_\Sigma(\Omega,\R^3)$ for every $j > N$. Given $f \in L^2_\Sigma(\Omega,\R^3)$ we define a quadratic functional $K \colon W^{1,2}_\Sigma(\Omega,\R^3) \to \R$ by
\[K(C) \defeq \frac{1}{2} \int_\Omega |\curl C(x)|^2 dx - \int_\Omega f(x) \cdot C(x) \, dx.\]
The quadratic part of $K$ is coercive, and therefore $K$ has a unique minimizer $w$ in $W^{1,2}_\Sigma(\Omega,\R^3)$. Thus
\begin{equation} \label{Definition of A2}
\int_\Omega \curl w(x) \cdot \curl \psi(x) \, dx = \int_\Omega f(x) \cdot \psi(x) \, dx
\end{equation}
for all $\psi \in W^{1,2}_\Sigma(\Omega,\R^3)$.

We define a bounded linear operator $\mathcal{A}_2 \colon L^2_\Sigma(\Omega,\R^3) \to L^2_\Sigma(\Omega,\R^3)$ by $\mathcal{A}_2 f \defeq w$. Our aim is to choose $w_j$, $j > N$, as eigenfunctions of $\mathcal{A}_2$. Since
\[\norm{w}_{W^{1,2}(\Omega)}^2 \lesssim_\Omega \norm{\curl w}_{L^2(\Omega)}^2 = \int_\Omega f \cdot w \le \norm{f}_{L^2(\Omega)} \norm{w}_{L^2(\Omega)}\]
for all $f \in L^2_\Sigma(\Omega,\R^3)$, the Rellich-Kondrachov Theorem implies that $\mathcal{A}_2$ is compact. In addition, $ \int_\Omega \mathcal{A}_2 f(x) \cdot g(x) \, dx = \int_\Omega \curl \mathcal{A}_2 f(x) \cdot \curl \mathcal{A}_2 g(x) \, dx = \int_\Omega f(x) \cdot \mathcal{A}_2 g(x) \, dx$ for all $f,g \in L^2_\Sigma(\Omega,\R^3)$ so that $\mathcal{A}_2$ is self-adjoint. Furthermore, $\mathcal{A}_2$ is a positive operator. Indeed, given $f \in L^2_\Sigma(\Omega,\R^3) \setminus \{0\}$ we write $\int_\Omega \mathcal{A}_2 f(x) \cdot f(x) \, dx = \norm{\curl \mathcal{A}_2 f}_{L^2(\Omega)}^2$. Now \eqref{Definition of A2} and the assumption $f \neq 0$ imply that $\curl \mathcal{A}_2 f \neq 0$: since $W^{1,2}_\Sigma(\Omega,\R^3)$ is dense in $L^2_\Sigma(\Omega,\R^3)$ by Lemma \ref{Orthogonal decomposition lemma}, we may choose $\psi \in W^{1,2}_\Sigma(\Omega,\R^3)$ such that $\int_\Omega \curl \mathcal{A}_2 f(x) \cdot \curl \psi(x) \, dx = \int_\Omega f(x) \cdot \psi(x) \, dx \neq 0$.

The Spectral Theorem for compact self-adjoint operators now yields an orthonormal basis $\{w_j\}_{j \in \N}$ of $L^2_\Sigma(\Omega,\R^3)$ and corresponding strictly positive eigenvalues $\mu_j \to 0$. We denote $\tilde{\lambda}_j \defeq 1/\mu_j \to \infty$. Equality \eqref{Magnetostatic equation} implies that the mappings $w_j/\sqrt{\tilde{\lambda_j}}$ form an orthonormal system in $W^{1,2}_\sigma(\Omega,\R^3)$.
\end{proof}

When $u \in W^{1,2}_{0,\sigma}(\Omega,\R^3)$, $b \in W^{1,2}_\sigma(\Omega,\R^3)$ and $n \in \N$, Lemmas \ref{Stokes lemma} and \ref{Magnetostatic lemma} allow us to write $P_n u$ and $Q_n b$ as
\[P_n u = \sum_{j=1}^n \left( u, \frac{v_j}{\sqrt{\lambda_j}} \right)_{W^{1,2}_{0,\sigma}} \frac{v_j}{\sqrt{\lambda_j}}, \quad Q_n b = \sum_{j=1}^n \left(b, \frac{w_j}{\sqrt{\tilde{\lambda}_j}} \right)_{W^{1,2}_\sigma} \frac{w_j}{\sqrt{\tilde{\lambda}_j}}.\]
This immediately implies the following result, which in turn yields the norm bound in \eqref{Uniform bound on operator norms}.

\begin{prop} \label{Proposition on projections}
Both of the linear operators $P_n \colon W^{1,2}_{0,\sigma}(\Omega,\R^3) \to W^{1,2}_{0,\sigma}(\Omega,\R^3)$ and $Q_n \colon W^{1,2}_\sigma(\Omega,\R^3) \to W^{1,2}_\sigma(\Omega,\R^3)$ are self-adjoint and bounded uniformly in $n$. 
\end{prop}

In the next subsection we give a solution of the Galerkin approximation equations.

\subsection{The Galerkin approximation}
In order to smoothen the exposition we work with the bases constructed in the previous subsection, although the following lemma holds for any orthonormal bases $\{v_j\}_{j \in \N}$ and $\{w_j\}_{j \in \N}$ of $L^2_\sigma(\Omega,\R^3)$ with $v_j \in W^{1,2}_{0,\sigma}(\Omega,\R^3)$ and $w_j \in W^{1,2}_\sigma(\Omega,\R^3)$.

\begin{lem} \label{Lemma on Galerkin approximations}
For every $n \in \N$, the Galerkin approximation has a solution of the form
\eqref{u_n and b_n} with the energy equality \eqref{Energy equality} holding for all $t \in [0,T)$.
\end{lem}

\begin{proof}
When $u_n$ and $b_n$ are of the form \eqref{u_n and b_n}, Lemmas \ref{Stokes lemma} and \ref{Magnetostatic lemma} imply that the Galerkin equations read as
\begin{align}
& \dot{c}_{nj}(t) - \nu \lambda_j c_{nj}(t) + \sum_{k,l=1}^n c_{nk}(t) c_{nl}(t) \alpha_{jkl} - \sum_{k,l=1}^n d_{nk}(t) d_{nl}(t) \beta_{jkl} = 0, \label{ODE 1} \\
& \dot{d}_{nj}(t) - \mu \tilde{\lambda}_j \chi_{j > N} d_{nj}(t) + \sum_{k,l=1}^n c_{nk}(t) d_{nl}(t) \gamma_{jkl} - \sum_{k,l=1}^n d_{nk}(t) c_{nl}(t) \delta_{jkl} = 0, \label{ODE 2} \\
& c_{nj}(0) = \int_\Omega u_0(x) \cdot v_j(x) \, dx, \quad d_{nj}(0) = \int_\Omega b_0(x) \cdot w_j(x) \, dx \label{ODE 3}
\end{align}
for $j = 1,\ldots,n$, where
\begin{align*}
& \alpha_{jkl} \defeq \int_\Omega (v_k(x) \cdot \nabla) v_l(x) \cdot v_j(x) \, dx, \quad \beta_{jkl} \defeq \int_\Omega (w_k(x) \cdot \nabla) w_l(x) \cdot v_j(x) \, dx \\
& \gamma_{jkl} \defeq \int_\Omega (w_k(x) \cdot \nabla) v_l(x) \cdot w_j(x) \, dx, \quad \delta_{jkl} \defeq \int_\Omega (v_k(x) \cdot \nabla) w_l(x) \cdot w_j(x) \, dx.
\end{align*}
Note that \eqref{ODE 1}--\eqref{ODE 3} is an initial value problem for a system of $2n$ ODE's on the $2n$ functions $c_{nj}, d_{nj}$, and by standard theory of ODE's there exists $T_n > 0$ and a solution $c_{n1},\ldots,d_{nn} \in C^\infty([0,T_n))$. Note also that
\begin{equation} \label{Values of coefficients}
\alpha_{jkl} = - \alpha_{lkj}, \quad \beta_{jkl} = -\gamma_{lkj}, \qquad \delta_{jkl} = - \delta_{lkj} = 0.
\end{equation}
The energy equality can be written as
\[\begin{array}{lcl}
& & \displaystyle \sum_{j=1}^n c_{nj}(t)^2 + \sum_{j=1}^n d_{nj}(t)^2 + 2 \sum_{j=1}^n \int_0^t (\nu \lambda_j c_{nj}(\tau)^2 + \mu \tilde{\lambda}_j \chi_{j > N} d_{nj}(\tau)^2) \, d\tau \\
&=& \displaystyle \sum_{j=1}^n c_{nj}(0)^2 +  \sum_{j=1}^n d_{nj}(0)^2
\end{array}\]
and is proved by multiplying \eqref{ODE 1} by $c_{nj}(t)$ and \eqref{ODE 2} by $d_{nj}(t)$, summing in $j$, integrating in time and using \eqref{Values of coefficients}. The energy equality allows us to continue the solution to \eqref{ODE 1}--\eqref{ODE 3} to the whole interval $[0,T)$.
\end{proof}

\subsection{Passing to the limit} \label{Passing to the limit}
The Leray-Hopf solution $(u,b)$ of \eqref{Resistive MHD}--\eqref{Resistive MHD6} will be obtained as a strong $L^2$ limit of $(u_n,b_n)$ by using the Aubin-Lions Lemma. To that end we prove norm bounds on $(u_n,b_n)$.

\begin{lem} \label{Lemma for Aubin-Lions}
There exists $C > 0$ such that
\begin{align*}
& \norm{u_n}_{L^\infty(0,T;L^2(\Omega))} + \norm{u_n}_{L^2(0,T;W^{1,2}(\Omega))} + \norm{\partial_t u_n}_{L^{4/3}(0,T;(W^{1,2}_{0,\sigma}(\Omega))^*)} \le C, \\
& \norm{b_n}_{L^\infty(0,T;L^2(\Omega))} + \norm{b_n}_{L^2(0,T;W^{1,2}(\Omega))} + \norm{\partial_t b_n}_{L^{4/3}(0,T;(W^{1,2}_\sigma(\Omega))^*)} \le C
\end{align*}
for all $n \in \N$.
\end{lem}

\begin{proof}
First, since $u_n$ and $b_n$ satisfy the energy equality for every $n \in \N$, $\|P_n u_0 - u_0\|_{L^2} \to 0$ and $\|Q_n b_0 - b_0\|_{L^2} \to 0$, it follows that $\sup_{n \in \N} (\|u_n\|_{L^\infty(0,T;L^2(\Omega))} + \|b_n\|_{L^\infty(0,T;L^2(\Omega))} ) < \infty$. By another use of the energy equality and Lemma \ref{Equivalent inner product}, $|\nabla u_n|$ and $|\nabla b_{n,\Sigma}|$ are uniformly bounded in $L^2(0,T;L^2(\Omega))$. Furthermore,
\begin{align*}
\|\nabla b_{n,H}\|_{L^2(0,T;L^2(\Omega,\R^{3 \times 3}))}
&= \left\| \sum_{i=1}^{\min{n,N}} d_{nj} \nabla w_j \right\|_{L^2(0,T;L^2(\Omega,\R^{3 \times 3}))} \\
&\lesssim_\Omega \left\| \sum_{i=1}^{\min{n,N}} d_{nj} w_j \right\|_{L^2(0,T;L^2(\Omega,\R^{3 \times 3}))} \\
&= \|b_{n,H}\|_{L^2(0,T;L^2(\Omega,\R^{3 \times 3}))}
\end{align*}
for all $n \in \N$, and thus $\sup_{n \in \N} (\norm{u_n}_{L^2(0,T;W^{1,2}(\Omega))} + \norm{b_n}_{L^2(0,T;W^{1,2}(\Omega))}) < \infty$.

We now deal with $\partial_t b_n$, $\partial_t u_n$ being similar but slightly simpler. At a.e. $t \in [0,T)$ and for all $\theta \in W^{1,2}_\sigma(\Omega,\R^3)$ we write
\[\langle \partial_t b_n, \theta \rangle_{(W^{1,2}_\sigma(\Omega))^*-W^{1,2}_\sigma(\Omega)} + \langle \mu \Lambda_2 b_n + \curl(b_n \times u_n), Q_n \theta \rangle_{(W^{1,2}_\sigma(\Omega))^*-W^{1,2}_\sigma(\Omega)} = 0.\]
Proposition \ref{Proposition on projections} gives
\begin{align*}
\langle \mu \Lambda_2 b_n, Q_n \theta \rangle_{(W^{1,2}_\sigma(\Omega))^*-W^{1,2}_\sigma(\Omega)}
&= \mu (\curl b_n(\cdot,t), \curl Q_n \theta)_{L^2} \\
&\lesssim \mu \norm{b_n(\cdot,t)}_{W^{1,2}_\sigma(\Omega)} \norm{\theta}_{W^{1,2}_\sigma(\Omega)}.
\end{align*}
By using Proposition \ref{Proposition on projections} again, given $n \in \N$ and $\theta \in W^{1,2}_\sigma(\Omega,\R^3)$ we get
\begin{align*}
|\langle \curl(b_n \times u_n), Q_n \theta \rangle_{(W^{1,2}_\sigma(\Omega))^*-W^{1,2}_\sigma(\Omega)}|
&\le \|b_n(\cdot,t)\|_{L^6} \|u_n(\cdot,t)\|_{L^3} \|\curl Q_n \theta\|_{L^2} \\
&\lesssim_\Omega \|\nabla b_n(\cdot,t)\|_{L^2} \|u_n(\cdot,t)\|_{L^3} \|\theta\|_{W^{1,2}_\sigma},
\end{align*}
so that, by using the previous inequality and H\"{o}lder's inequality with exponents $3/2$ and $3$ in $t$,
\begin{align*}
   \|\curl(u_n \times b_n)\|_{L^{4/3}(0,T;(W^{1,2}_{0,\sigma}(\Omega,\R^3))^*)}
&\lesssim_\Omega \|\norm{\nabla b_n}_{L^2(\Omega,\R^3)} \norm{u_n}_{L^3(\Omega,\R^3)}\|_{L^{4/3}(0,T)} \\
&\le \norm{\nabla b_n}_{L^2(0,T;L^2(\Omega,\R^{3 \times 3}))} \norm{u_n}_{L^4(0,T;L^3(\Omega,\R^3))}
\end{align*}
which, when combined with Lemma \ref{Interpolation lemma}, completes the proof.
\end{proof}

The Aubin-Lions Lemma and interpolation give various convergence properties.

\begin{lem} \label{Lemma on weak limits}
There exist $u \in L^\infty(0,T;L^2_\sigma(\Omega,\R^3)) \cap L^2(0,T;W^{1,2}_{0,\sigma}(\Omega,\R^3))$ and $b \in L^\infty(0,T;L^2_\sigma(\Omega,\R^3)) \cap L^2(0,T;W^{1,2}_\sigma(\Omega,\R^3))$ such that, up to a subsequence, the following convergences hold:
\renewcommand{\labelenumi}{(\roman{enumi})}
\begin{enumerate}
\item $u_n \to u$ and $b_b \to b$ in $L^2(0,T;L^2_\sigma(\Omega,\R^3))$,

\item $u_n \rightharpoonup u$ in $L^2(0,T;W^{1,2}_{0,\sigma}(\Omega,\R^3))$ and $b_n \rightharpoonup b$ in $L^2(0,T;W^{1,2}_\sigma(\Omega,\R^3))$,

\item $\partial_t u_n \rightharpoonup \partial_t u$ in $L^{4/3}(0,T;(W^{1,2}_{0,\sigma}(\Omega,\R^3)^*)$ and furthermore $\partial_t b_n \rightharpoonup \partial_t b$ in $L^{4/3}(0,T;(W^{1,2}_\sigma(\Omega,\R^3)^*)$,

\item $u_n \otimes u_n \rightharpoonup u \otimes u$ and $b_n \otimes b_n \rightharpoonup b \otimes b$ in $L^{4/3}(0,T;L^2(\Omega,\R^{3 \times 3}))$, $b_n \times u_n \rightharpoonup b \times u$ in $L^{4/3}(0,T;L^2(\Omega,\R^3))$.
\end{enumerate}
\end{lem}

\begin{proof}
While (i) and (ii) follow immediately from the Aubin-Lions Lemma and Lemma \ref{Lemma for Aubin-Lions}, claims (iii)--(iv) follow from (i) and Lemma \ref{Interpolation lemma}. The claim $u,b \in L^\infty(0,T;L^2_\sigma(\Omega,\R^3))$ follows from the fact that up to a subsequence, $u_n \overset{*}{\rightharpoonup} u$ and $b_n \overset{*}{\rightharpoonup} b$ in $L^\infty(0,T;L^2_\sigma(\Omega,\R^3))$.
\end{proof}

We show that $(u,b)$ solves the equations \eqref{Resistive MHD}--\eqref{2D resistive MHD3} and \eqref{Resistive MHD5}--\eqref{Resistive MHD6}, and we refer to ~\cite{Gal} for the proof of the claims that $u,b \in C_w([0,T);L^2_\sigma(\Omega,\R^3))$ and that $u(\cdot,0) = u_0$ and $b(\cdot,0) = b_0$. The energy inequality is then obtained as a consequence.

\begin{lem} \label{Lemma on equivalent definitions of weak solution}
The mappings $u$ and $b$ mentioned in Lemma \ref{Lemma on weak limits} form a Leray-Hopf solution of \eqref{Resistive MHD}--\eqref{Resistive MHD6}.
\end{lem}

\begin{proof}
We first show that $u$ and $b$ satisfy \eqref{Resistive MHD weak definition 1} and \eqref{Resistive MHD weak definition 2a} a.t a.e. $t \in [0,T)$ for every $\varphi \in W^{1,2}_{0,\sigma}(\Omega,\R^3)$ and $\theta \in W^{1,2}_\sigma(\Omega,\R^3)$. Note that whenever $\eta \in C_c^\infty([0,T))$ and $k \in \N$, Lemmas \ref{Lemma on Galerkin approximations} and \ref{Lemma on weak limits} give
\begin{align*}
   0
&= - \eta(0) (P_n u_0,v_k)_{L^2} - \int_0^T \eta'(t) (u_n,v_k)_{L^2} \, dt \\
&- \int_0^T \eta(t) (u_n \otimes u_n - b_n \otimes b_n, \nabla v_k)_{L^2} \, dt + \nu \int_0^T \eta(t) (\nabla u_n, \nabla v_k)_{L^2} \, dt \\
&\to - \eta(0) (u_0,v_k)_{L^2} - \int_0^T \eta'(t) (u,v_k)_{L^2} \, dt \\
&- \int_0^T \eta(t) (u \otimes u - b \otimes b, \nabla v_k)_{L^2} \, dt + \nu \int_0^T \eta(t) (\nabla u, \nabla v_k)_{L^2} \, dt,
\end{align*}
Given any $\varphi \in W^{1,2}_{0,\sigma}(\Omega,\R^3)$ we can replace $v_k$ above by $\varphi_k \defeq P_k \varphi$ by taking linear combinations. Now $\norm{P_k \varphi - \varphi}_{L^2(\Omega,\R^3)} \to 0$ and $\sup_{k \in \N} \norm{P_k \varphi}_{W^{1,2}_{0,\sigma}(\Omega,\R^3)} < \infty$ imply that $P_k \varphi \rightharpoonup \varphi$ in $W^{1,2}_{0,\sigma}(\Omega,\R^3)$. We let $k \to \infty$ to obtain
\begin{align*}
0 &= - \eta(0) (u_0,\varphi)_{L^2} - \int_0^T \eta'(t) (u,\varphi)_{L^2} \, dt \\
&- \int_0^T \eta(t) (u \otimes u - b \otimes b, \nabla \varphi)_{L^2} \, dt + \nu \int_0^T \eta(t) (\nabla u, \nabla \varphi)_{L^2} \, dt,
\end{align*}
which in particular gives \eqref{Resistive MHD weak definition 1} at a.e. $t \in [0,T)$. Similarly, if $\theta \in W^{1,2}_\sigma(\Omega,\R^3)$, equation \eqref{Resistive MHD weak definition 2a} holds at a.e. $t \in [0,T)$.

The claims that $u,b \in C_w([0,T);L^2_\sigma(\Omega,\R^3))$ and that $u(\cdot,0) = u_0$ and $b(\cdot,0) = b_0$ can be proved by slightly modifying \cite[Lemmas 2.1--2.2]{Gal}. Since we have $u,b \in C_w([0,T);L^2_\sigma(\Omega,\R^3))$, it suffices to show the energy inequality at a.e. $t \in [0,T)$. Since $u_n \to u$ and $b_n \to b$ in $L^2(0,T;L^2(\Omega,\R^3))$, passing to a subsequence we get $u_n(\cdot,t) \to u(\cdot,t)$ and $b_n(\cdot,t) \to b(\cdot,t)$ in $L^2(\Omega,\R^3)$ at a.e. $t \in [0,T)$. At those times $t$ the energy inequality for $u$ and $b$ now follows from the energy  equality of $u_n$ and $b_n$.
\end{proof}

\bibliography{TAYLOR}
\bibliographystyle{amsplain}

\end{document}